\documentclass[11pt]{amsart}
\usepackage{amssymb}
\usepackage{amsmath}
\usepackage{amsfonts}
\usepackage{graphicx}

\usepackage[total={17cm,22cm},top=2.5cm, left=2.3cm]{geometry}
\usepackage{hyperref}
    \usepackage{aeguill}
    \usepackage{type1cm}

\theoremstyle{plain}
\newtheorem{thm}{Theorem}[section]

\newtheorem{corollary}[thm]{Corollary}

\newtheorem{example}[thm]{Example}

\newtheorem{lemma}[thm]{Lemma}

\newtheorem{problem}[thm]{Problem}
\newtheorem{proposition}[thm]{Proposition}
\newtheorem{remark}[thm]{Remark}

\newtheorem{theorem}[thm]{Theorem}

\numberwithin{equation}{section}
\newcommand{\N}{\mathbb{N}}

\newcommand{\R}{\mathbb{R}}

\newcommand{\norm}[1]{\left\lVert #1 \right\rVert}

\DeclareMathOperator{\lip}{Lip\,\!}

\DeclareMathOperator{\sign}{sign\,\!}

\DeclareMathOperator{\co}{co\,\!}
\DeclareMathOperator{\conv}{conv\,\!}

\usepackage[usenames,dvipsnames]{color}

\usepackage[dvipsnames]{xcolor}

\newcommand{\ud}[0]{\,\mathrm{d}}

\begin{document}
\title[Differentiable convex extensions with sharp Lipschitz constants]{Differentiable convex extensions with sharp Lipschitz constants}

\author[Thomas Deck]{Thomas Deck}
\address[T.D.]{Department of Mathematical Sciences, Norwegian University of Science and Technology, 7941 Trondheim, Norway}
\email{tfdeck@ntnu.no}

 \author[Carlos Mudarra]{Carlos Mudarra}
\address[C.M.]{Department of Mathematical Sciences, Norwegian University of Science and Technology, 7941 Trondheim, Norway}
\email{carlos.mudarra@ntnu.no}

\date{\today}

\makeatletter
\@namedef{subjclassname@2020}{{\mdseries 2020} Mathematics Subject Classification}
\makeatother




\begin{abstract}
Given a superreflexive Banach space $X$, and a set $E \subset X$, we characterise the $1$-jets $(f,G)$ on $E$ that admit $C^{1,\omega}$ convex extensions $(F,DF)$ to all of $X$; where $\omega$ is any admissible modulus of continuity depending on the regularity of $X$. Moreover, we obtain precise estimates for the growth of the $C^{1,\omega}$ seminorm of the extensions with respect to the initial data. We show how these estimates can be improved in the Hilbert setting, and are asymptotically sharp for H\"older moduli. Remarkably, our extensions have the sharp Lipschitz constant $\mathrm{Lip}(F,X) = \|G\|_{L^\infty(E)}$, when $G$ is a bounded map. All these extensions are given by simple and explicit formulas. We also prove a similar theorem for $C^1$ convex extensions of jets defined on compact subsets $E$ of superreflexive spaces $X$, with the sharp Lipschitz constant too. The results are new even when $X=\mathbb{R}^n.$
\end{abstract}

\maketitle

\section{Introduction and main results}

By a modulus of continuity $\omega$, we understand a non-decreasing and concave function $\omega : [0,+\infty) \to [0,+\infty)$ with $\omega(0)=\omega(0^+)=0$. 
Given two normed spaces $(X, \|\cdot\|_X)$, $(Y, \| \cdot \|_Y)$, a subset $E \subset X$, and a map $H: E \to Y$, we denote 
$$
\lip_\omega(H,E):= \sup_{ x,z\in E, \, x \neq z} \frac{\|H(x)-H(z)\|_Y}{\omega(\|x-z\|_X)}.
$$
For H\"older modulus $\omega(t)=t^\alpha$, $\alpha \in (0,1],$ we will abbreviate by $\lip_\omega(H,E)=\lip_\alpha(H,E)$. And for Lipschitz functions, that is, when $\omega(t)=t$, we will simply write $\lip_\omega(H,E)= \lip(H,E).$

If $X$ is a normed space, we will denote by $X^*$ its dual space. Then $C^1(X)$ is the class of real-valued functions $F:X \to \R$ that are everywhere Fréchet differentiable in $X$ and have continuous Fréchet derivative $DF : X \to X^*$ in $X.$ And for a modulus of continuity $\omega$, the class $C^{1,\omega}(X)$ consists of functions $F\in C^1(X)$ so that $DF: X \to X^*$ is uniformly continuous with modulus $\omega;$ meaning that 
$
\lip_\omega(DF,X) < \infty. 
$
For functions with H\"older (resp. Lipschitz) derivatives, we will use the notation $C^{1,\alpha}(X)$ (resp. $C^{1,1}(X)$).

\smallskip

This paper concerns the following general Whitney-type problem for convex functions.

\begin{problem}\label{prob:general}
Let $X$ be Banach space, $E \subset X$ a set, $f :E \to \R$, $G:E \to X^*$ two functions, and $\omega$ a modulus of continuity. 

What are the necessary and sufficient conditions on $(f,G)$ that would guarantee the existence of a \emph{convex} function $F: X \to \R$ of class $C^1(X)$ or $C^{1,\omega}(X)$ so that $F=f$ and $DF =G$ on $E$?

In the case of $C^{1,\omega}(X)$ extensions $F$, what would be quasi-sharp estimates for $\lip_\omega(DF,X)$ in terms of the data $(f,G)$?

If $G$ is, in addition, a bounded map, can those extensions $F \in C^{1,\omega}(X)$ or $F\in C^1(X)$ be taken to be Lipschitz, with a quasi-sharp estimate for $\lip(F,X)$ in terms of $G$?

\end{problem}

Significant progress on the solution to Problem \ref{prob:general} has been made in the last few years, especially concerning the first two questions. In the case $X=\R^n$ and $E \subset \R^n$ a compact set, precise necessary and sufficient conditions to guarantee extendability of jets $(f,G)$ in $E$ by $C^1(\R^n)$ convex functions were found in \cite{AM17}. Those extensions are Lipschitz, with bounds of the form $\lip(F,\R^n) \leq C(n) \sup_E |G|$, for a dimensional constant $C(n).$ This was extended to the Hilbert space $X$, for a compact $E \subset X$, with the bound $\lip(F, X) \leq 5 \sup_E |G|.$ In the much more complicated case $X=\R^n$, $E \subset \R^n$ arbitrary, and the class $C^1(\R^n)$, a full answer to the first question was given in \cite{AM19APDE}. Moreover, in the case of bounded maps $G :E \to \R^n$, the extensions $F$ were Lipschitz, with an estimate of the type $\lip(F,\R^n) \leq C(n) \sup_E |G|$, for a dimensional constant $C(n).$

In a Hilbert space $X$, and $E \subset X$, and the $C^{1,1}$ class, the results in \cite{ALeGM18} provide a complete answer to the first two questions, by means of simple and explicit formulas for the extensions $F.$ Moreover, these extensions have the sharp Lipschitz constant $\lip(\nabla F,X)$ for the gradient $\nabla F$. A similar result for $C^{1,\omega}$ extensions was obtained in \cite[Theorem 4.11]{ALeGM18}, with a bound for $\lip_\omega(\nabla F,X)$ in terms of a suitable constant associated with $(f,G)$, and a multiplicative factor $C=8.$

Moreover, in the case where $X$ is a superreflexive Banach space admitting a renorming with modulus of smoothness of power type $1+\alpha,$ $\alpha\in (0,1]$, the full answer to the first question for $E$ arbitrary and the class $C^{1,\alpha}(X)$ was given in \cite[Theorem 5.5]{ALeGM18}.

The third question in Problem \ref{prob:general} concerning Lipschitz-condition preservation (in the case of bounded maps $G$) was not addressed in \cite{ALeGM18}. 
 
\medskip

Our contributions in this paper are as follows:

\begin{itemize}
    \item In a superreflexive Banach space $X$ admitting a renorming of smoothness $1+\alpha,$ we prove an extension theorem for $C^{1,\omega}$ convex functions, where $\omega$ is \textit{any} modulus of continuity that is \textit{dominated} by $\alpha,$ thus generalising \cite[Theorem 5.5]{ALeGM18}.
    \item In Hilbert spaces $X$, we significantly improve the estimates for the $\omega$-seminorm $\lip_\omega(DF,X)$ of the derivatives of the extensions $F$ obtained in \cite[Theorem 4.11]{ALeGM18}, for any modulus of continuity. These estimates can be further improved in the $\alpha$-H\"older case, and they are \textit{asymptotically sharp}, in the sense that they converge to $1$ when $\alpha \to 1^-.$ 
    \item We prove that, when $G$ is bounded on $E,$ one can construct convex extensions $F \in C^{1,\omega}(X)$ with the same quasi-optimal bounds for $\lip_\omega(DF,X)$, and the additional feature that $F$ \textit{has the sharp Lipschitz constant}:
    $$
    \lip(F, X) = \sup_{z\in E} \|G(z)\|_*.
    $$
    This is achieved via a \textit{Lipschitz-convex envelope}, for which we have found explicit formulas as an infimum of convex combinations of parabolas defined by the initial jet. 
    \item We solve Problem \ref{prob:general} for $C^1$ regularity in the case where $X$ is a superreflexive space and $E \subset X$ is compact; in particular, extending \cite[Theorem 1]{AM20} from the Hilbert to the superreflexive setting. Our extensions $F$ have the sharp Lipschitz constant $\lip(F)=\sup_{z\in E} \|G(z)\|_*$ as well.
\end{itemize}

To our knowledge, the parts concerning the sharp Lipschitz constant are new even when $X=\R^n$, both for $C^1$ and $C^{1,\omega}$ extensions.

\medskip

The fundamental starting point for this type of problem is the famous Whitney Extension Theorem \cite{W34}, which provides necessary and sufficient conditions for jets defined on subsets of $\R^n$ to admit an extension of class $C^m(\R^n)$. The first version of this theorem for $C^{m,\omega}$ functions appeared in the work of Glaeser \cite{G58}, along with the solution of the corresponding problem for functions (instead of jets) for the class $C^1(\R^n).$ A sharp form of Whitney-Glaeser theorem for $C^{1,1}$ was obtained in \cite{W73}, and then an alternate and much shorter proof of this theorem appeared in \cite{LeG09}. However, as a consequence of the solution to the corresponding problem for $C^{1,1}$ convex functions, it was possible to obtain the sharp $C^{1,1}$ theorem for general functions via simple and explicit formulas; see \cite[Theorem 3.4]{ALeGM18}. The first infinite-dimensional Whitney-type theorem for $C^{1,\omega}$ arbitrary functions was recently obtained in \cite{AM21}. We also refer to \cite{JSSG13} for the first version of the Whitney's theorem for the class $C^1$ in Hilbert and more general Banach spaces, including the vector-valued setting. See also \cite{JZ25} for recent results on smooth extensions between Banach spaces, including $C^1$ extensions of functions from quasiconvex open sets.

For Whitney-type extensions of jets generated by Sobolev functions in $\R^n$, see, for
instance, the recent work \cite{S17}. And for
extension results for functions (instead of jets) of order $C^m$ and $C^{m,\omega}$ and for Sobolev functions, we refer to the papers \cite{BS01}, \cite{F05,F06}, \cite{FIL14}. This list is by no means exhaustive.

\medskip

Let us now describe the necessary conditions and formulas to state our results. Given a modulus of continuity $\omega,$ we define the associated integral function
$$
\varphi_\omega(t) = \int_0^t \omega(s) \ud s, \quad t\geq 0. 
$$ 

If $\emptyset \neq E \subset X$ is a set, a $1$-jet is a couple $(f,G),$ with $f: E \to \R$ and $G : E \to X^*.$ We define the \textit{condition \eqref{eq:definitionCw1omega_forallx} with constant $M>0$} by
\begin{equation}\label{eq:definitionCw1omega_forallx}
f(y)  + G(y)(x-y) + M \varphi_\omega(\|x-y\|) \geq f(z) + G(z)(x-z) , \quad \text{for all} \quad y,z\in E, \, x\in X. \tag{$\widetilde{CW^{1,\omega}}$}
\end{equation}
In order to gauge the optimal $M>0$ in this condition we introduce the following constant:
\begin{equation}\label{eq:definitionseminormA}
A_{\omega, \co}(f,G,E) : = \sup  \left\lbrace \frac{f(z)+G(z)(x-z)-f(y)-G(y)(x-y)}{\varphi_\omega(\|x-y\|)} \, : \, y,z\in E, \, x\in X, \ x \neq y \right\rbrace,
\end{equation}
for any jet $(f,G) : E \to \R \times X^*.$ Note that $0 \leq A_{\omega, \co}(f,G,E) \leq +\infty$ and that $A_{\omega,\co}(f,G,E)$ is the smallest $M \geq 0$ for which $(f,G)$ satisfies the condition \eqref{eq:definitionCw1omega_forallx} with constant $M$ on $E.$ For convex functions $F\in C^{1,\omega}(X)$ one has
$$
A_{\omega, \co}(F,DF,X) \leq \lip_\omega(DF,X);
$$
see Lemma \ref{lem:necessityofCW1omega}. See also Theorem \ref{thm:relationsAseminorm_Lip} below for a precise reverse inequality for jets, which happens to be optimal for H\"older modulus.

\smallskip

If, in addition, $\omega$ is increasing and coercive (that is, $\lim_{t \to \infty} \omega(t)= \infty$), an alternate and very useful reformulation of \eqref{eq:definitionCw1omega_forallx} is the following condition \eqref{eq:definitionCw1omega} with constant $M>0:$
\begin{equation}\label{eq:definitionCw1omega}
f(y) \geq f(z) + G(z)(y-z) + M (\varphi_\omega)^* \left( \frac{1}{M} \|G(y)-G(z)\|_* \right), \quad \text{for all} \quad y,z\in E; \tag{$CW^{1,\omega}$}
\end{equation}
where $(\varphi_\omega)^*$ is the Fenchel conjugate of the function $\varphi_\omega.$ It was shown in \cite{AM17} that \eqref{eq:definitionCw1omega} is the necessary and sufficient condition to extend jets with convex $C^{1,\omega}$ functions in Hilbert spaces, and in superreflexive spaces for some H\"older modulus.

Observe that, unlike for \eqref{eq:definitionCw1omega_forallx}, the condition \eqref{eq:definitionCw1omega} only involves points $y,z\in E$. This \textit{intrinsic} character makes the condition \eqref{eq:definitionCw1omega} easily testable. Interestingly, it will be shown in Lemma \ref{lem:rewriteCW1omega} that these conditions are fully identical.

\smallskip

The \textit{convex envelope} $\conv(g)$ of a function $g:X \to \R$ with a continuous affine minorant in $X$ is
\begin{equation}\label{eq:definition_convexenvelope}
\conv(g)(x)  = \sup\{ h(x) \, : \, h : X \to \R \, \textrm{ is convex and continuous, } \, h\leq g \text{ on } X\}   \quad x\in X. 
\end{equation}

Alternate well-known formulas for $\conv(g)$ are collected in Section \ref{sect:formulae_convexenvelopes}.

\smallskip
 
Moreover, if $L>0$ and $g: X \to \R$ has an $L$-Lipschitz affine minorant in $X,$ we define the \textit{convex and $L$-Lipschitz envelope of $g$} by:
\begin{equation}\label{eq:definition_convex_L_envelope}
\conv_L(g)  =  \sup\{ h(x) \, : \, h : X \to \R \, \textrm{ is convex and $L$-Lipschitz, } \, h\leq g \text{ on } X\}   \quad x\in X.
\end{equation}

In Section \ref{sect:formulae_convexenvelopes}, Lemma \ref{lem:formulae_convL} we will give two alternate formulas for $\conv_L(g).$

\medskip

Now, let $(X, \| \cdot \|)$ be a superreflexive Banach space. By Pisier's renorming theorem \cite[Theorem 3.1]{P75}, there exists an equivalent norm in $X$ (which we keep denoting by $\|\cdot \|$) with \textit{modulus of smoothness of power type} $1+\alpha$. This means that there is $C>0$ with
\begin{equation}\label{eq:modulussmoothnessalpha}
\lambda \|x\|^{1+\alpha} + (1-\lambda)\|y\|^{1+\alpha} -  \|\lambda x + (1-\lambda)y\|^{1+\alpha} \leq C \cdot \lambda (1-\lambda)  \|x-y\|^{1+\alpha}, \quad \lambda \in [0,1], \, x,y \in X.
\end{equation}

Our first main result is as follows.

\begin{theorem}\label{thm:mainsuperreflexive}
Assume that $(X,\| \cdot\|)$ satisfies \eqref{eq:modulussmoothnessalpha} and let $\omega$ be a modulus of continuity with the property that
\begin{equation}\label{eq:assumptiononOmegaAlpha}
t \mapsto \frac{t^\alpha}{\omega(t)} \quad \text{is non-decreasing}.
\end{equation}
Let $E \subset X$ and $(f,G)$ a jet in $E.$ Then there exists $F\in C^{1,\omega}(X)$ convex with $(F,D F) = (f,G)$ on $E$ if and only if $(f,G)$ satisfies the condition \eqref{eq:definitionCw1omega_forallx} for some $M>0.$ Moreover, the formula
\begin{equation}\label{eq:formulaextensionC1wsuperreflexive}
F:= \conv(g), \quad g(x) = \inf_{y\in E} \lbrace f(y) + G(y)(x-y) +M  \varphi_\omega(\|x-y\|) \rbrace, \quad x\in X
\end{equation}
defines such an extension, and
\begin{equation}\label{eq:estimate_extensionoperator_C1wsuperreflexive}
A_{\omega, \co}(F,DF,X) \leq C(\alpha)   \cdot C \cdot M \quad \text{and} \quad \lip_\omega(DF,X) \leq \frac{4}{3} \cdot C(\alpha) \cdot C \cdot M;
\end{equation}
where $C$ is that of \eqref{eq:modulussmoothnessalpha}, and $C(\alpha)>0$ depends only on $\alpha.$

Moreover, if in addition $G: E \to X^*$ is bounded and $L:=\sup_{z\in E} \|G(z)\|_*$, then the function
\begin{equation}\label{eq:formulaF_LinTheoremsuperreflexive}
F_L:= \conv_L(g)    
\end{equation}
has the same properties as those of $F$, with the additional feature that
\begin{equation}\label{eq:mainthereomsuperrefl_Lipschitzpresserved}
  \sup_{x\in X} \|DF_L(x) \|_* = \lip(F_L) = L = \sup_{z\in E} \|G(z)\|_*
\end{equation}

Furthermore, if $\omega$ is an increasing modulus satisfying $\lim_{t \to \infty} \omega(t)=\infty,$ then all the above statements are valid with condition \eqref{eq:definitionCw1omega} with constant $M>0$ in place of \eqref{eq:definitionCw1omega_forallx}. 
\end{theorem}

We can improve Theorem \ref{thm:mainsuperreflexive} when $X$ is a Hilbert space with the following theorems. 

\begin{theorem}\label{thm:mainHilbert}
Let $(X, | \cdot|)$ be a Hilbert space and let $\omega$ be a modulus of continuity that is increasing and with $\lim_{t \to \infty} \omega(t)=\infty.$ Let $E \subset X$ and $(f,G) : E \to \R \times X$ be a jet in $E.$ Then there exists $F\in C^{1,\omega}(X)$ convex with $(F, \nabla F) = (f,G)$ on $E$ if and only if $(f,G)$ satisfies the condition \eqref{eq:definitionCw1omega} for some $M>0.$ Moreover, in that case, the formula
\begin{equation}\label{eq:formulaextensionC1wHilbert}
F:= \conv(g), \quad g(x) = \inf_{y\in E} \lbrace f(y) + \langle G(y) , x-y \rangle + M \varphi_\omega(|x-y|) \rbrace, \quad x\in X
\end{equation}
defines such an extension, and
\begin{equation}\label{eq:estimate_extensionoperator_C1wHilbert}
A_{\omega,\co}(F,\nabla F, X) \leq 2  M \quad \text{and} \quad \lip_\omega(\nabla F, X) \leq \frac{8}{3} M.
\end{equation}
Moreover, if in addition $G: E \to X^*$ is bounded and $L:=\sup_{z\in E} |G(z)|$, then the function $F_L:= \conv_L(g)$ has the same properties as those of $F$, with the additional property that $F_L$ is Lipschitz on $X$ with
$$
  \sup_{x\in X} | \nabla F_L(x) | = \lip(F_L) = L = \sup_{z\in E} |G(z)|.
$$

\end{theorem}

For functions with H\"older derivatives, Theorem \ref{thm:mainHilbert} admits the following improvement. 

\begin{theorem}\label{thm:mainHilbertHoldercase}
Let $0 < \alpha \leq 1.$ Let $X$ be a Hilbert space, let $E \subset X$, and $(f,G) : E \to \R \times X$ be a jet in $E.$ Then there exists $F\in C^{1,\alpha}(X)$ convex with $(F, \nabla F) = (f,G)$ on $E$ if and only if $(f,G)$ satisfies the condition
\begin{equation}\label{eq:definitionCW1alpha}
f(y) \geq f(z) +  \langle G(z) , y-z \rangle + \frac{\alpha}{(1+\alpha) M^{1/\alpha}} |G(y)-G(z)|^{1 + \frac{1}{\alpha}}, \quad \text{for all} \quad y,z \in E, \tag{$CW^{1,\alpha}$}
\end{equation}
for some $M>0.$ 
Moreover, in that case, the formula
\begin{equation}\label{eq:formulaextension_Hilbert_Holder}
F:= \conv(g), \quad g(x) = \inf_{y\in E} \lbrace f(y) + \langle G(y) , x-y \rangle + \frac{M}{1+\alpha}|x-y|^{1+\alpha} \rbrace, \quad x\in X
\end{equation}
defines such an extension, and
\begin{equation}\label{eq:estimate_extensionoperator_C1wHilbert_Holdercase}
A_{\alpha,\co}(F,\nabla F, X) \leq 2^{1-\alpha}  M \quad \text{and} \quad \lip_\alpha(\nabla F, X) \leq 2^{1-\alpha} \left(\frac{1 + \alpha}{2\alpha}\right)^\alpha M.
\end{equation}
Furthermore, if in addition $G: E \to X^*$ is bounded and $L:=\sup_{z\in E}  |G(z) | $, then the function $F_L:= \conv_L(g)$ has the same properties as those of $F$, with the additional property that $F_L$ is Lipschitz on $X$ with
$$
  \sup_{x\in X} | \nabla F_L(x) | = \lip(F_L) = L = \sup_{z\in E} |G(z)|.
$$
\end{theorem}

\begin{remark}\label{rem:takingthesmallestM}
{\em 
Theorems \ref{thm:mainsuperreflexive}, \ref{thm:mainHilbert} and \ref{thm:mainHilbertHoldercase} can be applied with the constant $M=A_{\omega, \co}(f,G,E),$ thus obtaining convex extensions $F \in C^{1,\omega}(X)$ of jets $(f,G) :E \to E \times X^*$ with the estimates
$$
A_{\omega, \co}(F,DF,X) \leq C(\alpha)   \cdot C \cdot A_{\omega,\co}(f,G,E) \quad \text{and} \quad \lip_\omega(DF,X) \leq \frac{4}{3} \cdot C(\alpha) \cdot C \cdot A_{\omega, \co}(f,G,E);
$$
$$
A_{\omega,\co}(F,\nabla F, X) \leq 2  A_{\omega,\co}(f,G,E) \quad \text{and} \quad \lip_\omega(\nabla F, X) \leq \frac{8}{3} A_{\omega,\co}(f,G,E);
$$
$$
A_{\omega,\co}(F,\nabla F, X) \leq 2^{1-\alpha}  A_{\alpha,\co}(f,G,E) \quad \text{and} \quad  \lip_\alpha(\nabla F, X) \leq 2^{1-\alpha} \left(\frac{1 + \alpha}{2\alpha}\right)^\alpha A_{\omega,\co}(f,G,E)
$$
respectively.  
}    
\end{remark}

As we mentioned earlier, in the case where $G$ is bounded, the parts of Theorems \ref{thm:mainsuperreflexive}, \ref{thm:mainHilbert}, \ref{thm:mainHilbertHoldercase} concerning the sharp Lipschitz constant are new even when $X=  \R^n$. Note that here we use a new extension formula via the \textit{$L$-Lipschitz and convex envelope} $F_L=\conv_L(g)$ in \eqref{eq:formulaF_LinTheoremsuperreflexive}; for which we prove more explicit formulas in Lemma \ref{lem:formulae_convL}. 
The preservation of possible Lipschitz constants for $C^{1,\omega}$ convex extensions was not considered in \cite{ALeGM18}. 
\\
In the case of arbitrary (possibly unbounded) $G$, we are using the same type of extension formulas \eqref{eq:formulaextensionC1wsuperreflexive}, \eqref{eq:formulaextensionC1wHilbert} as in \cite[Theorem 4.11]{ALeGM18}. But now Theorem \ref{thm:mainsuperreflexive} is valid for all modulus $\omega$ satisfying \eqref{eq:assumptiononOmegaAlpha} with respect to $\alpha$, thus generalising \cite[Theorem 5.5]{ALeGM18}, which was valid only for the class $C^{1,\alpha}$ itself. Condition \eqref{eq:assumptiononOmegaAlpha} includes the cases $\omega(t) = t^\beta,$ $0 <\beta \leq \alpha.$
\\
Also, Theorems \ref{thm:mainHilbert} and \ref{thm:mainHilbertHoldercase} have much better constants \eqref{eq:estimate_extensionoperator_C1wHilbert}, \eqref{eq:estimate_extensionoperator_C1wHilbert_Holdercase} than their antecedent in \cite[Theorem 4.11]{ALeGM18}. In fact, Theorem \ref{thm:mainHilbertHoldercase} implies \cite[Theorem 2.4]{ALeGM18}, as the bounds \eqref{eq:estimate_extensionoperator_C1wHilbert_Holdercase} equal $1$ when $\alpha=1.$ 
\\
We will prove Theorems \ref{thm:mainsuperreflexive}, \ref{thm:mainHilbert}, \ref{thm:mainHilbertHoldercase} by first proving a more general result in Banach spaces for $C^{1,\omega}$-smoothness; see Theorem \ref{thm:technicalgeneraltheorem}.

\medskip

We have also obtained an optimal result for convex extensions of order $C^1(X).$ In a Banach space $(X , \| \cdot\|)$, let $E \subset X$ and $(f,G): E \to \R \times X^*$ a $1$-jet. We define the following two conditions for $(f,G)$ on $E:$ 
\begin{equation}\label{eq:condition(C)}
f(y) \geq f(z) + G(z)(y-z) , \quad \text{for all} \quad y,z\in E. \tag{$C$}
\end{equation}
\begin{equation}\label{eq:condition(CW1)}
f(y) = f(z) + G(z)(y-z) \: \text{ implies } \: G(y)=G(z), \quad y,z\in E. \tag{$CW^1$}
\end{equation}
These conditions were shown to characterise the extendability of a jet $(f,G)$ in a compact set $E$ by $C^1$ and convex extensions when $X= \R^n$ \cite{AM17} or when $X$ is a Hilbert space \cite{AM20}.

\smallskip 

In a superreflexive space $X,$ we have obtained the following result.

\begin{theorem}\label{thm:C1optimal}
Let $X$ be a superreflexive Banach space, let $E \subset X$ be compact and let $(f,G): E \to \R \times X^*$ be a $1$-jet in $E$ with $G$ continuous. Then there exists a convex function $F\in C^1(X)$ with $(F,DF)=(f,G)$ on $E$ if and only if $(f,G)$ satisfies the conditions \eqref{eq:condition(C)} and \eqref{eq:condition(CW1)}.

Moreover, $F$ can be taken to be Lipschitz on $X$, with
$$
\lip(F) = \sup_{x\in X} \|DF(x)\|_* = \sup_{z\in E} \|G(z)\|_*.
$$
\end{theorem}

The first part of Theorem \ref{thm:C1optimal} extends \cite[Theorem 1]{AM20} from Hilbert to superreflexive Banach space. And, again, the result about the sharp Lipschitz constant is new even when $X= \R^n,$ since the best bounds obtained so far were, by different means, $\lip(F) \leq 5  \sup_{z\in E} |G(z)|$ (see \cite{AM20}) or $\lip (F) \leq C(n)\sup_{z\in E} |G(z)|$, where $C(n)$ is a dimensional constant (see \cite{AM17}). 
\\
It is also worth comparing Theorem \ref{thm:C1optimal} with the corresponding results for $C^1$ general (not necessarily convex) extensions in terms of the Lipschitz properties. The best bounds obtained so far \cite{JSSG13} for the Lipschitz constants of an $F \in C^1(X)$ extending $(f,G)$ from a subset $E$ (in the case where $G$ is bounded), are of the form 
$$
\lip(F_\varepsilon,X) \leq  C_0 \cdot \max\lbrace \lip(f,E), L \rbrace + \varepsilon, \quad \text{where} \quad L =  \sup_{z\in E} \|G(z)\|_*,
$$
for arbitrarily small (but fixed) $\varepsilon>0,$ and $C_0 \geq 1$ a constant depending on the $X$. When $X$ is a Hilbert space, the constant $C_0$ can taken equal to $1.$ Remarkably, Theorem \ref{thm:C1optimal} for convex functions provides extensions without increasing the \textit{natural} Lipschitz constant of the jet.

\smallskip

Note that in our case, $L:=\sup_{z\in E} \|G(z)\|_*$ is the right magnitude to gauge the \textit{Lipschitz constant} of a jet $(f,G)$ admitting a convex extension, because the convexity condition \eqref{eq:condition(C)} implies that
$$
f(z)-f(y) \leq G(z)(y-z) \leq L \|y-z\|, \quad \text{and} \quad f(y)-f(z) \leq G(y)(z-y) \leq L \|z-y\|, \quad y,z\in E, 
$$
and so $\lip(f,E) \leq L.$

\medskip

The structure of the paper is as follows. In Section \ref{sect:preliminaries}, we collect some basic properties of Fenchel conjugates and moduli of continuity, as well as some useful inequalities concerning the smoothness of Hilbertian norms. In Section \ref{sect:relationsbetweennorms} we establish several relationships between the constants $\lip_\omega$ and $A_{\omega, \co}$ for restricted jets and globally defined convex functions, and we also prove the equivalences between \eqref{eq:definitionCw1omega} and \eqref{eq:definitionCw1omega_forallx}. It is precisely these estimates that allow us to obtain precise bounds for the seminorms of the extended functions in Theorems \ref{thm:mainsuperreflexive}, \ref{thm:mainHilbert}, and \ref{thm:mainHilbertHoldercase}. Then, in Section \ref{sect:formulae_convexenvelopes} we prove various formulae for the Lipschitz and convex envelopes $\conv_L$, based on which we establish the pertinent $C^{1,\omega}$-regularity results. We will prove a key technical Theorem \ref{thm:technicalgeneraltheorem} in Section \ref{sect:generalresult_prooffirstmaintheorems}, which will be used to derive Theorems \ref{thm:mainsuperreflexive}, \ref{thm:mainHilbert}, and \ref{thm:mainHilbertHoldercase} in the same section. Finally, in Section \ref{sect:C1case}, we give the proof of Theorem \ref{thm:C1optimal}.

\section{Preliminaries and tools from convex analysis}\label{sect:preliminaries}
In this section we recall some basic results from convex analysis, as well as certain inequalities concerning modulus of continuity $\omega$ and their integral function $\varphi_\omega$.

For a function $\delta : [0, \infty) \to [0, \infty),$ the \textit{Fenchel conjugate} of $\delta$ is defined by
$$
\delta^*(t) = \sup\lbrace ts-\delta(s) \, : \, s\in \R \rbrace, \quad s\in \R.
$$

We will mostly consider the Fenchel conjugate of nonnegative functions only defined on $[0,+\infty),$ say $\delta: [0,+\infty) \to [0,+\infty)$, slightly abusing of terminology. In such case, we will always assume that all the functions involved are extended to all of $\R$ by setting $\delta(t)= \delta(-t)$ for $t<0$. Hence $\delta$ will be an even function on $\R$ so that its conjugate $\delta^*$ satisfies
$$
\delta^*(t)= \sup_{ s\in \R} \lbrace ts-\delta(s) \rbrace =  \sup_{ s \geq 0} \lbrace ts-\delta(s) \rbrace ,\quad \text{for} \quad t \geq 0.
$$

We refer to \cite[Section 3.2]{BV10} and \cite[Section 3.2]{Z02} for a detailed presentation of Fenchel conjugates and their great importance in convex analysis. 

\smallskip

We will make use of the following properties and estimates for functions associated with a modulus $\omega,$ i.e., a concave non-decreasing function on $[0,\infty)$ with $\omega(0)=\omega(0^+)=0$.

\begin{proposition}\label{prop:inequalities_omega_varphiomega}
Let $\omega:[0, \infty)\to [0,\infty)$ be a modulus of continuity, and let $\varphi_\omega(t)=\int_{0}^{t}w(s) \ud s$. Then:
\begin{enumerate}
\item $\omega( \lambda t)\leq \lambda \omega(t)$ for all $\lambda \geq 1, t\geq 0$, and the function $t \mapsto \frac{t}{\omega(t)}$ is non-decreasing.  
\item $(t/2) \omega(t) \leq \varphi_\omega(t) \leq t\omega(t/2)$ for all $t\geq 0$;
\end{enumerate}
If, in addition, $\omega$ is increasing and $\lim_{t \to \infty} \omega(t) = \infty,$ then $\omega^{-1}$ and $(\varphi_\omega)^*$ are well defined and
\begin{enumerate}
\item[$(3)$]  $(\varphi_\omega)^*(t)= \int_0^t \omega^{-1}(s) \ud s$ for all $t\geq 0$;
\item[$(4)$]  $(\varphi_\omega)(t)+(\varphi_\omega)^*(s)=ts$ if and only if $s= \omega(t);$ 
\item[$(5)$] $t\omega^{-1}(t/2)\leq (\varphi_\omega)^{*}(t)\leq (t/2)\omega^{-1}(t)$ for all $t\geq 0$.
 \end{enumerate}
\end{proposition}

Properties (1) and (2) easily follow from the concavity of $\omega.$ The identity (3) and property (4) are proven in \cite[Lemma 3.7.1]{Z02}, and (5) is an easy consequence of those.

\smallskip

The following lemma was already proved in \cite{AM21}, as a corollary of the results in \cite{VNC78}. We state here in a slightly different form, which will be convenient for our proofs in Sections \ref{sect:formulae_convexenvelopes} and \ref{sect:generalresult_prooffirstmaintheorems}.

\begin{lemma}\label{lem:regularity_psiomega_Hilbert}
Let $(X, |\cdot|)$ be a Hilbert space, and $\omega$ an modulus of continuity that is increasing and with $\lim_{t\to\infty}\omega(t)=\infty.$ Then the function $\psi_\omega(x)= \varphi_\omega(| x|), \: x\in X$, satisfies:
$$
 \lambda \psi_\omega(z + (1 - \lambda)h)+ (1-\lambda)\psi_\omega(z - \lambda h) - \psi_\omega(z) \leq  2 \lambda(1-\lambda)\varphi_\omega(|h|) \quad \text{for all} \quad z,h \in X, \: \lambda \in [0,1].
$$
Moreover, if $\omega(t) = t^\alpha,$ with $\alpha \in (0,1],$ the function $\psi_\alpha(x)= \varphi_\alpha(| x|), \: x\in X$, satisfies the inequality:
$$ 
 \lambda \psi_\alpha(z + (1 - \lambda)h)+ (1-\lambda)\psi_\alpha(z - \lambda h) - \psi_\alpha(z) \leq  2^{1-\alpha} \lambda(1-\lambda)\varphi_\alpha(|h|) \quad \text{for all} \quad z,h \in X, \: \lambda \in [0,1].
 $$
\end{lemma}
\begin{proof}
It was shown in \cite[Lemmas 2.3 and 2.4]{AM21} that $\psi_\omega$ satisfies the following inequality:
$$
\lambda \psi_\omega(x) + (1-\lambda) \psi(y) - \psi_\omega( \lambda x+ (1-\lambda)y) \leq 2 \lambda(1-\lambda) \psi_\omega(|x-y|), \quad x,y\in X, \, \lambda \in [0,1].
$$
This is actually a corollary of the results in \cite[Theorem 3]{VNC78} for uniformly convex functions, after applying an elementary duality argument. In order to get the inequality in the desired form, it suffices to replace $z$ with $\lambda x + (1-\lambda) y$ and $h$ with $x-y$ in the last inequality, since then
$$
z+(1-\lambda) h = x, \quad z-\lambda h = y.
$$

The H\"older case follows with the exact same argument. 

\end{proof}

\section{Smoothness and modulus of continuity for convex-type jets}\label{sect:relationsbetweennorms}

In this section, we establish precise relationships between the different magnitudes we will be using to gauge the $C^{1,\omega}$ regularity of our jets and functions.    

\subsection{Conditions and estimates for arbitrary jets}

As we announced in the introduction, the conditions \eqref{eq:definitionCw1omega_forallx} and \eqref{eq:definitionCw1omega} are identical, with the same constant $M>0,$ when the modulus $\omega$ is assumed to be increasing and coercive. This is the content of the following lemma.

\begin{lemma}\label{lem:rewriteCW1omega}
Let $X$ be a normed space, and $\omega$ an increasing modulus of continuity satisfying $\lim_{t \to \infty}\omega(t)=\infty$. Then, for a jet $(f,G)$ on a set $E \subset X,$ and $M>0,$ the condition \eqref{eq:definitionCw1omega} with constant $M$ is equivalent to \eqref{eq:definitionCw1omega_forallx} with the same constant. 
\end{lemma}
\begin{proof}
Assume first that $(f,G)$ satisfies \eqref{eq:definitionCw1omega} for $M>0$. Obviously $M \varphi_\omega = \varphi_{M \omega}.$ So, using elementary properties of Fenchel conjugates, namely, that 
$$
M (\varphi_\omega)^*(t/M)  = (M \varphi_\omega)^*(t) = (\varphi_{M \omega})^*(t)
$$
we see that \eqref{eq:definitionCw1omega} can be rewritten as
$$
f(y) \geq f(z) + G(z)(y-z) + (\varphi_{M\omega})^*\left(\norm{G(y) - G(z)}_*\right), \quad \text{for all} \quad y,z\in E.
$$
Now, if $x\in X,$ $y,z\in E,$ plugging in our inequality from \eqref{eq:definitionCw1omega} we have 
\begin{align*}
& f(y)  + G(y)(x-y) +  \varphi_{M\omega}(\|x-y\|) - f(z) - G(z)(x-z) \\
&\geq f(z) + G(z)(y-z) + (\varphi_{M\omega})^*(\norm{G(y)-G(z)}_* + G(y)(x-y) + \varphi_{M\omega}(\norm{x-y}) - f(z) - G(z)(x-z)\\
&= G(z)(y -x) + G(y)(x-y) +(\varphi_{M\omega})^*(\norm{G(y)-G(z)}_*) + \varphi_{M\omega}(\norm{x-y})\\
&= (G(y)-G(z))(x-y) + (\varphi_{M\omega})^*(\norm{G(y)-G(z)}_*) + \varphi_{M\omega}(\norm{x-y}).
\end{align*}
We can now use the Fenchel-Young inequality, that is, for $t,s \geq 0$ we have 
$$
ts \leq (\varphi_{M\omega})^*(t) + \varphi_{M\omega} (s);
$$
see Proposition \ref{prop:inequalities_omega_varphiomega}. Letting $t = \norm{G(y)-G(z)}_*$ and $s = \norm{x-y}$, we have 
\begin{align*}
  & (G(y)-G(z))(x-y)   +  (\varphi_{M\omega})^*(\norm{G(y)-G(z)}_*) + \varphi_{M\omega}(\norm{x-y}) \\
 & \quad \geq (G(y)-G(z))(x-y)  + \norm{G(y)-G(z)}_*\norm{x-y} \geq 0.
\end{align*}
We may conclude that 
$$
f(y)  + G(y)(x-y) +  \varphi_{M\omega}(\|x-y\|) - f(z) - G(z)(x-z) \geq 0,
$$
which is precisely condition \eqref{eq:definitionCw1omega_forallx}. 

\smallskip

Conversely, assume that $(f,G)$ satisfies condition \eqref{eq:definitionCw1omega_forallx} with constant $M>0.$ Fix $y,z\in E$, and let $\varepsilon>0.$ By the definition of the norm $\|G(y)-G(z)\|_*$ we can find $v \in X^*$ with
\begin{equation}\label{eq:proofCW1wequivalenceschoiceofv}
\|v\|= (M\omega)^{-1}( \|G(y)-G(z)\|_*) \quad \text{and} \quad (G(y)-G(z)(v) \geq \|G(y)-G(z)\|_* \|v\|- \varepsilon.    
\end{equation}
Define $x=y+v$ and write $s = \| G(y)-G(z)\|_*.$ By the properties of the Fenchel conjugate for the function $(M \varphi_\omega) = \varphi_{M \omega}$, we obtain
$$
s (M \omega)^{-1}(s) = ( \varphi_{M\omega})((M\omega)^{-1}(s)) + (\varphi_{M \omega})^* (s). 
$$
By the choice of $\|v\|$ (see \eqref{eq:proofCW1wequivalenceschoiceofv} this is the same as
\begin{align} \label{eq:identityFenchelconjugates_Gyzv}
   \|G(y)-G(z)\|_* \|v\| &  = (M \varphi_\omega)(\|v\|)+ (M \varphi_\omega)^*(\| G(y)-G(z)\|_*)  \nonumber \\
   & =  (M \varphi_\omega)(\|x-y\|)+ (M \varphi_\omega)^*(\| G(y)-G(z)\|_*) .
\end{align}
Using first \eqref{eq:definitionCw1omega_forallx} with $x$ as above, then \eqref{eq:proofCW1wequivalenceschoiceofv}, and finally \eqref{eq:identityFenchelconjugates_Gyzv} we get that
\begin{align*}
  f(y)& -f(z)-G(z)(y-z)-(M \varphi_\omega)^*(\|G(y)-G(z)\|_*) \\
  & = f(y)+G(y)(x-y)-f(z)-G(z)(x-z) + (G(z)-G(y))(x-y) -(M \varphi_\omega)^*(\|G(y)-G(z)\|_*) \\
  & \geq -(M \varphi_\omega)(\|x-y)\| +  \|G(y)-G(z)\|_* \|v\|- \varepsilon    -(M \varphi_\omega)^*(\|G(y)-G(z)\|_*) \\
  & \geq -\varepsilon.
\end{align*}
Since $\varepsilon>0$ is arbitrary, we deduce that
$$
f(y)- f(z)-G(z)(y-z)-(M \varphi_\omega)^*(\|G(y)-G(z)\|_*) \geq 0,
$$
thus obtaining the inequality \eqref{eq:definitionCw1omega}.
\end{proof}

Let us mention that that such an exact equivalence between a \textit{half-extrinsic} (like \ref{eq:definitionCw1omega_forallx}) and \textit{fully intrinsic} (like \ref{eq:definitionCw1omega_forallx}) has not been established for in the general (not necessarily convex) $C^{1,\omega}$ setting, and not even in the setting of Hilbert spaces and H\"older modulus of continuity. See for example conditions $(W^{1,\omega})$ and $(mg^{1,\omega})$ in \cite{AM21}, which are shown to be equivalent but \textit{not} with the same constant $M>0.$

\medskip

We now precisely estimate $\lip_\omega(G,E)$ in terms of $A_{\omega, \co}(f,G,E)$ with absolute multiplicative factors, for arbitrary jets $(f,G) : E \to \R \times X^*.$

\begin{theorem}\label{thm:relationsAseminorm_Lip}
Let $(X,\|\cdot\|)$ be a normed space, $E \subset X$, a $1$-jet $(f,G) : E \to \R \times X^*$, and a modulus of continuity $\omega$. Then we have the estimate
$$
\lip_\omega(G,E) \leq \frac{4}{3} A_{\omega, \co}(f,G,E).
$$
Moreover, in the case $\omega(t) = t^\alpha,$ with $0 < \alpha \leq 1,$ we can arrange
\begin{equation}\label{eq:estimateLipalphaAalpha}
    \lip_\alpha(G,E) \leq \left(\frac{1 + \alpha}{2\alpha}\right)^\alpha A_{\alpha, \co}(f,G,E).
\end{equation}
\end{theorem}
\begin{proof}
We may assume that $0 < A_{\omega, \co}(f,G,E) < \infty$, as otherwise, the estimates trivially hold. By homogeneity, we may also assume that $A_{\omega, \co}(f,G,E)=1.$

Now, given $x,x' \in X$, $y,z \in E$ we have the two following inequalities 
\begin{align*}
    f(y) + G(y)(x-y) - f(z) - G(z)(x-z) &\geq - \varphi_\omega\left(\norm{x-y}\right)\\
    f(z) + G(z)(x'-z) - f(y) - G(y)(x'-y) &\geq  -\varphi_\omega(\norm{x-y})
\end{align*}
Combining them we arrive at
$$
(G(y) - G(z))(x'-x) \leq \varphi_\omega\left(\norm{x-y}\right) + \varphi_\omega\left(\norm{x'-z}\right)
$$
Since our choices of $x$ and $x'$ were arbitrary, set them as $x = \frac{1}{2}(y+z)+v$ and $x' = \frac{1}{2}(y+z)-v$, with $v\in X$ arbitrary. We then have 
$$
(G(y) - G(z))(-v) \leq \varphi_\omega\left(\norm{\frac{1}{2}(y-z) + v}\right) \leq \varphi_\omega\left(\norm{\frac{1}{2}(y-z)} + \norm{v}\right)
$$
Writing $v = t\norm{y-z}u$ with $u\in X, \, \norm{u} = 1$, $t>0,$ and taking the supremum over $\|u\|=1$, the above estimate implies 
$$
\norm{G(y) - G(z)}_* \leq \frac{\varphi_\omega\left(\left(\frac{1}{2}+t)\norm{y-z}\right)\right)}{t\norm{y-z}}, \quad t>0, 
$$
that is
\begin{equation}\label{eq:estimateLipomegawithinfimumovert}
\norm{G(y) - G(z)}_* \leq \inf_{t > 0}\left\{\frac{\varphi_\omega\left(\left(\frac{1}{2}+t)\norm{y-z}\right)\right)}{t\norm{y-z}}\right\} .
\end{equation}
Applying the inequality $\varphi_\omega(t) \leq t \omega(t/2)$ from Proposition \ref{prop:inequalities_omega_varphiomega}, the above leads us to 
\begin{align*}
\norm{G(y) - G(z)}_* &  \leq \inf_{t > 0}\left\{ \left( 1 + \frac{1}{2t} \right) \omega \left( \frac{1}{2}( \frac{1}{2}+t) ) \|y-z\| \right)  \right\} \\
& \leq \inf_{ 0 < t \leq 3/2}\left\{ \left( 1 + \frac{1}{2t} \right) \omega(\|y-z\| \right \}  = \frac{4}{3} \omega(\|y-z\|).    
\end{align*}
This proves the first inequality.

Now, let $\omega(t)=t^\alpha$, so that $\varphi_\omega(t) = \frac{1}{1 + \alpha}t^{1 + \alpha}$. Inserting this into \eqref{eq:estimateLipomegawithinfimumovert} we get
$$
\norm{G(y) - G(z)}_* \leq \inf_{t > 0}\left\{\frac{\frac{1}{1+\alpha}\left((\frac{1}{2}+t)\norm{y-z}\right))^{1+\alpha}}{t\norm{y-z}}\right\}.
$$
To compute the infimum, we note that the function $h(t) = ( \frac{1}{2} + t )^{1+\alpha}/t$, for $t>0,$ attains the minimum at $t_0= 1/(2\alpha)$. Substituting this $t_0$ we obtain the inequality
$$
   \norm{G(y) - G(z)}_*  \leq \frac{(1 + \alpha)^{1 + \alpha}}{(2\alpha)^\alpha}\frac{\norm{y-z}^\alpha}{(1+\alpha)} = \left(\frac{1 + \alpha}{2\alpha}\right)^\alpha\norm{y-z}^\alpha,
$$
as desired. 
\end{proof}

Interestingly, the inequality \eqref{eq:estimateLipalphaAalpha} is sharp, as shown by following example. 

\begin{example}\label{ex:optimality_relationseminorms_Holder}
{\em
Let $X=\R$ with the usual metric, and let $E = \{-1, 1\}$. Define $F : \R \to \R$ by
$$
    F(t)  = \frac{|t|^{1 + \alpha}}{1 + \alpha}, \:  \text{ so that } \: F'(t)  =  \sign(t) |t|^\alpha, \quad t\in \R. 
$$
Define the jet $(f, G)=(F, F')|_{E},$ the restriction of $(F,F')$ to $E.$ We claim that
\begin{equation}\label{eq:sharpinequality_example}
\lip_\alpha(G,E) = \left(\frac{1 + \alpha}{2\alpha}\right)^\alpha A_{\alpha, \co}(f,G,E).
\end{equation}
thus showing that the inequality \eqref{eq:estimateLipalphaAalpha} is sharp. Indeed, clearly
$$
\lip_\alpha(G,E) = \sup_{y,z\in E, \, y \neq z}\frac{|G(y) - G(z)|}{|y - z|^\alpha}  = \frac{|G(1) - G(-1)|}{|1 - (-1)|^\alpha} = \frac{2}{2^\alpha} = 2^{1 - \alpha}.
$$ 
On the other hand, according to Lemma \ref{lem:rewriteCW1omega}, $A_{\alpha, \co}(f,G,E)$ is the smallest $M>0$ for which
$$
f(t) \geq f(s) + G(s)(t-s) + \frac{1}{1 + \frac{1}{\alpha}} |G(t) - G(s)|^{1+ \frac{1}{\alpha}}  M^{-1/\alpha}, \quad \text{for all} \quad t, s\in E,
$$
after using that, in this case, $(\varphi_\alpha)^* (t)= \frac{1}{1 + \frac{1}{\alpha}}t^{1+\frac{1}{ \alpha}},$ for $t \geq 0.$ 
Since $f(\pm 1) = 1/(1+\alpha)$ and $G( \pm 1 ) = \pm 1,$ we see that, either when $y=1,$ $z=-1$ or $y=-1,$ $z=1,$ those numbers $M>0$ must satisfy the inequality
$$
2 M^{1/\alpha} \geq \frac{1}{1 + \frac{1}{\alpha}} 2^{1+ \frac{1}{\alpha}}.
$$
The infimum of those is $M = \frac{2}{\left(1 + \frac{1}{\alpha}\right)^\alpha} ,$ and therefore
$$
A_{\alpha,\co}(f,G,E) = \frac{2}{\left(1 + \frac{1}{\alpha}\right)^\alpha}  .
$$
We immediately see that then \eqref{eq:sharpinequality_example} holds. 
}
\end{example}

\medskip

On the other hand, for jets $(F,DF)$ associated with globally defined convex functions $F: X \to \R$, the formula \eqref{eq:definitionseminormA} can be simplified as follows.

\begin{lemma}\label{lem:seminormAforextendedfunctions}
Let $(X, \|\cdot \|)$ be Banach space and $F\in C^{1,\omega}(X)$ be a convex function. Then
$$
A_{\omega, \co}(F,DF,X) =   \sup  \left\lbrace \frac{F(x)-F(y)-DF(y)(x-y)}{\varphi_\omega(\|x-y\|)} \, : \, x,y\in X, \, x \neq y \right\rbrace.
$$
\end{lemma}
\begin{proof}
 Indeed, denoting by $A$ the supremum of the right-hand side, clearly $A \leq A_{\omega, \co}(F,DF,X), $ as we may take $z=x$ in the definition of $A_{\omega, \co}(F,DF,X)$; see \eqref{eq:definitionseminormA}. The convexity of $F$ gives the reverse inequality, as 
$$
F(z)+ DF(z)(x-z)- F(y)-DF(y)(x-y) \leq F(x) -F(y)-DF(y)(x-y)
$$
for all $x,y,z\in X.$   
\end{proof}

Finally, observe that even for convex functions $F \in C^{1,\omega}(X)$, the quantity $A_{\alpha , \co}(F,DF,X)$ may be strictly smaller than $\lip_{1/2} (DF, X)$. Indeed, letting $f(t) = \frac{2}{3} |t|^{3/2},$ $t\in \R,$ and $\alpha = 1/2,$ one has that
$
\lip_\alpha(f',\R) = \sqrt{2}.
$
And, according to Lemma \ref{lem:seminormAforextendedfunctions}, we have that
$$
A_{1/2, \co}(f,f',\R)= \sup  \left\lbrace \frac{f(t)-f(s)-f'(s)(t-s)}{\frac{2}{3}|t-s|^{3/2}} \, : \, t,s\in \R, \, t \neq s \right\rbrace.
$$
By homogeneity, it is easily seen that
$$
A_{1/2, \co}(f,f',\R)= \sup  \left\lbrace \frac{t^{3/2}+3 t^{1/2}+2}{2(t+1)^{3/2}} \, : \, t>0  \right\rbrace \leq 1,3066 < \sqrt{2} = \lip_{1/2}(f', \R).
$$

\subsection{Uniform smoothness and $C^{1,\omega}$ regularity for convex functions}

It is a well-known result that a convex and lower semicontinous function $F :X \to \R$ in a Banach space $X$ is of class $C^{1,\omega}(X)$ if and only $F$ satisfies the uniform smoothness condition
$$
F(x+h) + F(x-h) - 2 F(x) \leq C \|h\| \omega(\|h\|), \quad x,h \in X.
$$
A proof of this fact can be essentially found in \cite[Lemma V.3.5]{DGZ93} or in \cite[Proposition 4.5]{ALeGM18}. Nevertheless, since we are interested in quasi-sharp estimates for our constants $\lip_\omega(DF,X)$ and $A_{\omega, \co}(F,DF,X)$, we next reformulate this result in the following lemma, and provide the proof that leads to those precise estimates. 
 
\begin{lemma}\label{lem:convexcombregularity_impliesC1omega}
Let $X$ be a Banach space, $F: X \to \R$ a lower semicontinuous convex function, and $\omega$ a modulus of continuity. Assume that there is a constant $A>0$ so that
\begin{equation}\label{eq:unfiformsmoothnessF_inBanach}
\lambda F(x+(1-\lambda) h) + (1-\lambda) F(x-\lambda h) - F(x) \leq   \lambda (1-\lambda) \cdot A  \cdot \varphi_\omega(\|h\|), \quad \lambda \in [0,1], \, x,h \in X. 
\end{equation}
Then $F\in C^{1,\omega}(X)$ with $A_{\omega, \co}(F,DF, X) \leq A$ and $\lip_\omega(DF,X) \leq \frac{4}{3}A.$ Moreover, if $\omega(t) = t^\alpha,$ one can achieve the estimate $\lip_\alpha(DF,X) \leq \left( \frac{1+\alpha}{2\alpha} \right)^\alpha A.$ 
\end{lemma}
\begin{proof}
Note that letting $\lambda=1/2$ and replacing $h$ with $2h,$ \eqref{eq:unfiformsmoothnessF_inBanach} gives
$$
 F(x+h) +  F(x- h) - F(x) \leq  \frac{A}{2}  \varphi_\omega(\|2h\|), \quad  x,h \in X. 
$$
Diving by $\|h\|$ in this inequality, and using Proposition \ref{prop:inequalities_omega_varphiomega}, we get 
$$
 \frac{F(x+h) +  F(x- h) - F(x)}{\|h\|} \leq \frac{A}{2} \frac{\varphi_\omega(2\|h\|)}{\|h\|} \leq A \omega(\|h\|),
$$
where the last term tends to $0$ as $\|h\| \to 0.$ Because $F$ is convex and lower semicontinuous in $X$, this implies that $F$ is Fréchet differentiable at every $z\in X;$ see \cite[Proposition 4.5]{ALeGM18} for a proof of this well-known result.

Let us now prove that $A_{\omega, \co}(F,DF,X) \leq A.$ Dividing by $\lambda \in (0,1)$ in \eqref{eq:unfiformsmoothnessF_inBanach}, we get the inequality
$$
F(x+(1-\lambda)h)-F(x-\lambda h ) + \frac{1}{\lambda}\left( F(x-\lambda h)-F(z) \right) \leq (1-\lambda) \cdot A \cdot \varphi_\omega(\|h\|),
$$
for all $\lambda \in (0,1),$ $x,h\in X.$ Letting $\lambda \to 0,$ the differentiability (and continuity) of $F$ in $X$ gives
$$
F(x+h)-F(x)-DF(x)(h) \leq A \cdot \varphi_\omega(\|h\|), \quad x,h\in X.
$$
In other words,
$$
F(x) \leq F(y) + DF(y)(x-y) + A \cdot \varphi_\omega(\|x-y\|), \quad x,y\in X.
$$
The convexity of $F$ then leads us to
$$
F(z)+ DF(z)(x-z) \leq F(x) \leq  F(y) + DF(y)(x-y) + A \cdot \varphi_\omega(\|x-y\|), \quad x,y,z\in X.
$$
This clearly shows that $A_{\omega,\co}(F,DF,X) \leq A.$ Finally, the estimate for $\lip_\omega(DF,X)$ in terms of $A$ is a direct consequence of Theorem \ref{thm:relationsAseminorm_Lip}. 
\end{proof}

\section{Lipschitz and convex envelopes: formulas and regularity}\label{sect:formulae_convexenvelopes}

In this section we prove various formulas for the Lipschitz and convex envelopes $\conv_L(g)$ we used in our theorems, and collect the (already known) formulas for the standard convex envelope $\conv(g).$

\medskip

Throughout this section, let $(X, \| \cdot\|)$ be a Banach space, let $\omega$ a modulus of continuity, let $L>0,$ and a function $g : X \to \R$ having a continuous convex minorant in $X.$ Let us denote
$$
F:= \conv(g), 
$$
as in \eqref{eq:definition_convexenvelope}. The following formulas for $\conv(g)$ are well-known:
\begin{align}\label{eq:threeformulas_conv}
F(x) :  &  =  \sup\{ h(x) \, : \, h : X \to \R \, \textrm{ is convex and continuous, } \, h\leq g \text{ on } X\} \nonumber \\
&   =  \sup\{ h(x) \, : \, h : X \to \R \, \textrm{ is affine and continuous, } \, h\leq g \text{ on } X\} \nonumber \\
& = \inf\left\lbrace \sum_{j=1}^{n}\lambda_{j} g(x_j) \, : \, \lambda_j\geq 0,
\: \sum_{j=1}^{n}\lambda_j =1, \, x_j \in X, \: x=\sum_{j=1}^{n}\lambda_j x_j, \: n\in\N \right\rbrace, \quad x\in X.
\end{align}
The first two formulas are the same simply because every continuous convex functions can be written as a sumpremum of affine and continuous functions. For the identity between the first and third formula, it is sufficient to notice that term in \eqref{eq:threeformulas_conv} defines a convex function on $x\in X.$

In the special case $X= \R^n$, Carathéodory's Theorem implies that it is enough to consider convex combinations of at most $n+1$ points in formula \eqref{eq:threeformulas_conv}, and so
\begin{equation}\label{eq:formulaconv_inRn}
    F(x) = \inf\left\lbrace \sum_{j=1}^{n+1}\lambda_{j} g(x_j) \, : \, \lambda_j\geq 0,
\: \sum_{j=1}^{n+1}\lambda_j =1, \, x_j \in \R^n \: x=\sum_{j=1}^{n+1}\lambda_j x_j \right\rbrace, \quad x\in \R^n.
\end{equation}

\medskip

Now assume further that $g$ has an $L$-Lipschitz and convex minorant in $X$, and define
$$
F_L:= \conv_L(g)
$$
as in the definition \eqref{eq:definition_convex_L_envelope}. Using that every convex and $L$-Lipschitz function $h: X \to \R$ can be written as a supremum of affine and $L$-Lipschitz functions, we get that:
\begin{align*}
F_L(x) :   & =  \sup\{ h(x) \, : \, h : X \to \R \textrm{ is convex and $L$-Lipschitz, } h\leq g \text{ on } X\} \\
& = \sup\{ \ell(x) \, : \, \ell : X \to \R \textrm{ is affine and $L$-Lipschitz, } \ell \leq g \text{ on } X\},  \quad x\in X.
\end{align*}

However, we are interested in a formula similar to \eqref{eq:threeformulas_conv} for the Lipschitz and convex envelope $F_L$, which allows for a direct computation or approximation of the value $F_L(x)$ solely using the function $g.$ The desired formula is established in the next lemma. 

\begin{lemma}\label{lem:formulae_convL}
If $X$ is a Banach space, we have, for all $x\in X$, that
\begin{align*}
F_L(x) & = \inf\lbrace F(y) + L \|x-y\| \, : \, y\in X \rbrace \\
& =   \inf\left\lbrace L \|x-y\| + \sum_{j=1}^{n}\lambda_{j} g(y_j) \: : \: \lambda_j\geq 0, \:
\sum_{j=1}^{n}\lambda_j =1, \: y=\sum_{j=1}^{n}\lambda_j y_j, \: y_j \in X, \: n\in\N, \: y\in X \right\rbrace.
\end{align*}
\end{lemma}
\begin{proof}
Bearing in mind the infimum-type formula for $F(y) = \conv(g)(y)$ with convex combinations $\sum_{j=1}^n \lambda_j y_j=y$, the identity between the second and the third expressions is immediate. Therefore, we only need to show that
$$
F_L(x) = \inf\lbrace F(y) + L \|x-y\| \, : \, y\in X \rbrace : = H(x) , \quad x\in X.
$$
Let $m$ be an $L$-Lipschitz and convex minorant of $g$ in $X.$ In particular $m$ is convex and, by the definition of $F$ (as the largest convex function lying below $g$), we have $F \geq m$ everywhere in $X.$ Thus
$$
F(y) + L \|x-y\| \geq m(y) + L \|x-y\| \geq m(x),
$$
for all $x,y\in X.$ in the last inequality, we have used that $m$ is $L$-Lipschitz. This shows that $H(x) > - \infty$ for all $x\in X.$ Also, it is obvious that

Therefore the function $H$ is, as a finite infimum of $L$-Lipschitz functions, $L$-Lipschitz in $X$ as well. Let us now show that $H : X \to \R$ is convex in $X.$ Indeed, let $x,z\in X,$ and $\lambda \in [0,1].$ By the definitions of $H(x)$ and $H(z)$, for every $\varepsilon>0,$ we can find $y_x, y_z \in X$ with
$$
H(x) \geq F(y_x) + L\|x-y_x\| - \varepsilon, \quad H(z) \geq F(y_z)  + L\|z-y_z\| + \varepsilon.
$$
Considering the point $\lambda y_x + (1-\lambda) y_z \in X,$ the definition of $H(\lambda  x + (1-\lambda)  z)$, the convexity of $F$ in $X$ and the choices of $y_x$ and $y_z$ give
\begin{align*}
H(\lambda y_x + (1-\lambda) y_z) & \leq F(\lambda y_x + (1-\lambda) y_z ) + L\| \lambda  x + (1-\lambda)  z- (\lambda  y_x + (1-\lambda) y_z) \|\\
& \leq \lambda F(y_x) + (1-\lambda) F(y_z) + L \lambda \|x-y_x\| + L (1-\lambda) \|z-y_z\| \\
& \leq \lambda \left( F(y_x) + L\|x-y \| \right) + (1-\lambda) \left( F(y_z) + L\|z-y_z\| \right) \\
& \leq \lambda (H(x)+\varepsilon) + (1-\lambda) ( H(z) + \varepsilon ) \\
& = \lambda H(x) + (1-\lambda) H(z) + \varepsilon.
\end{align*}
Since $\varepsilon>0$ is arbitrary, the above shows that $H$ is convex in $X.$ Therefore, $H: X \to \R$ is convex and $L$-Lipschitz. Also, clearly from the definition of $H,$ one has $H(x) \leq F(x) \leq g(x).$ Thus, $H$ is one of the competitors in the definition of $F_L = \conv_L(g),$ implying that $F_L \geq H$ in $X.$

For the reverse inequality, note that $F_L$ is an $L$-Lipschitz and convex function, as a supremum of $L$-Lipschitz and convex functions. And obviously $F_L \leq F$ on $X$ by their definitions. Consequently,
$$
F_L(x) \leq F_L(y) + L \|x-y\| \leq F(y) + L\|x-y\|
$$
for all $x,y\in X.$ Taking the infimum over $y\in X$ in the right hand side, we conclude $F_L(x) \leq H(x),$ as desired.
\end{proof}

As a corollary of Lemma \ref{lem:formulae_convL} and \eqref{eq:formulaconv_inRn}, we get a simplified version in the finite-dimensional setting.

\begin{corollary}\label{cor:formulae_convL_inRn}
If $X= \R^n$, then the following formula holds for $F_L:$ 
$$ 
F_L(x) = \inf\left\lbrace L \|x-y\| + \sum_{j=1}^{n+1}\lambda_{j} g(y_j) \: : \: \lambda_j\geq 0, \:
\sum_{j=1}^{n+1}\lambda_j =1, \: y=\sum_{j=1}^{n+1}\lambda_j y_j, \: y_j, y \in \R^n \right\rbrace, \quad x\in \R^n.
$$
\end{corollary}

The formula of Lemma \ref{lem:formulae_convL} allows to prove the following regularity result for the Lipschitz and convex envelope. 

\begin{lemma}\label{lem:proofregularity_convL_withinfimumformula}
Let $\omega$ be a modulus of continuity and assume that there is $\kappa>0$ so that $g$ satisfies the estimate
$$
\lambda g(x+(1-\lambda)h) + (1-\lambda) g(x-\lambda h) - g(x) \leq  \lambda (1-\lambda) \cdot \kappa \cdot \varphi_\omega(\|h\|)
$$
for all $\lambda \in [0,1]$, $x,h \in X.$ Then the same estimate holds with $F_L$ in place of $g.$
\end{lemma}
\begin{proof}
We will use the last formula of Lemma \ref{lem:formulae_convL} for $F_L.$ Let $\lambda \in [0,1]$, $x,h \in X$. For every $\varepsilon>0,$ we can find $y\in X,$ $n\in \N,$ $\lbrace y_j \rbrace_{j=1}^n \subset X$, $\lbrace \lambda_j \rbrace_{j=1}^n \geq 0$ with $\sum_{j=1}^n \lambda_j=1$, $\sum_{j=1}^n \lambda_j y_j= y$ and
$$
F_L(x) \geq L\|x-y\| + \sum_{j=1}^n \lambda_j g(y_j) - \varepsilon.
$$
We have that $\sum_{j=1}^n \lambda_j(y_j-\lambda h) =y-\lambda h$ and $\sum_{j=1}^n \lambda_j(y_j+(1-\lambda) h) = y+ (1-\lambda)h.$ Therefore using the same $\inf$-type formula for $F_L(x-\lambda h)$ and $F(x+(1-\lambda)h),$ one has
$$
F_L(x-\lambda h) \leq L\| x- \lambda h - (y-\lambda h ) \| + \sum_{j=1}^n \lambda_j g(y_j -\lambda h) = L \|x-y\| + \sum_{j=1}^n \lambda_j g(y_j -\lambda h) ,
$$
\begin{align*}
F_L(x+(1-\lambda) h) & \leq L\| x+ (1-\lambda) h - (y+(1-\lambda) h ) \| + \sum_{j=1}^n \lambda_j g(y_j +(1-\lambda) h) \\
& = L \|x-y\| + \sum_{j=1}^n \lambda_j g(y_j +(1-\lambda) h).
\end{align*}
Putting all the pieces together and using the assumption on $g,$ we have
\begin{align*}
 \lambda F_L(x+(1-\lambda)h)   & + (1-\lambda) F_L(x-\lambda h) - F_L(x)   \\
& \leq \lambda  [ L \|x-y\| + \sum_{j=1}^n \lambda_j g(y_j -\lambda h)   ]  + (1-\lambda) [ L \|x-y\| + \sum_{j=1}^n \lambda_j g(y_j +(1-\lambda) h) ]  \\
& \quad  - L\|x-y\| - \sum_{j=1}^n \lambda_j g(y_j) + \varepsilon \\
& =\sum_{j=1}^n \lambda_ j \left[ \lambda g(y_j+(1-\lambda)h) + (1-\lambda) g(y_j-\lambda h) - g(y_j) \right] + \varepsilon \\
& \leq \sum_{j=1}^n \lambda_ j \kappa \varphi_\omega(\|h\|) + \varepsilon = \kappa \varphi_\omega(\|h\|) + \varepsilon.
\end{align*}
Because $\varepsilon>0$ is arbitrary, we get the desired estimate for $F_L.$

\end{proof}

\section{A key general result. Proofs of Theorems \ref{thm:mainsuperreflexive}, \ref{thm:mainHilbert}, and \ref{thm:mainHilbertHoldercase}}\label{sect:generalresult_prooffirstmaintheorems}

Theorems \ref{thm:mainsuperreflexive}, \ref{thm:mainHilbert}, and \ref{thm:mainHilbertHoldercase} are a consequence of a general extension theorem for $C^{1,\omega}$ convex functions in general Banach spaces, in combination with results on the regularity of the norm in the pertinent space. We now state and prove this general result, where the space is assumed to have certain $C^{1,\omega}$-smoothness.

\begin{theorem}\label{thm:technicalgeneraltheorem}
Let $(X,\| \cdot\|)$ be a Banach space, and let $\omega$ be a modulus of continuity so that the function $\psi_\omega = \varphi_\omega \circ \| \cdot \| : X \to \R$ satisfies
\begin{equation}\label{eq:assumptiononregularitypsiomega_generaltheorem}
\lambda \psi_\omega(x+(1-\lambda)h) + (1-\lambda) \psi_\omega(x-\lambda h) - \psi_\omega(z) \leq \lambda (1-\lambda) \cdot K \cdot \varphi_\omega(\|h\|), \quad \lambda \in [0,1], \, z,h \in X.
\end{equation}
Let $E \subset X$ and $(f,G)$ a jet in $E.$ Then there exists $F\in C^{1,\omega}(X)$ convex with $(F,D F) = (f,G)$ on $E$ if and only if $(f,G)$ satisfies the condition \eqref{eq:definitionCw1omega_forallx} for some $M>0.$ Moreover, the formula
\begin{equation}\label{eq:formulaextensionC1w_generaltheorem}
F:= \conv(g), \quad g(x) = \inf_{y\in E} \lbrace f(y) + G(y)(x-y) + M \varphi_\omega(\|x-y\|) \rbrace, \quad x\in X
\end{equation}
defines such an extension, and
\begin{equation}\label{eq:estimate_extensionoperator_generaltheorem}
A_{\omega, \co}(F,DF,X) \leq K \cdot M  \quad \text{and} \quad \lip_\omega(DF,X) \leq \frac{4}{3} K \cdot M;
\end{equation}
where $K$ is that of \eqref{eq:assumptiononregularitypsiomega_generaltheorem}.

Furthermore, if in addition $G: E \to X^*$ is bounded and $L:=\sup_{z\in E} \|G(z)\|_*$, then the function
$$
F_L:= \conv_L(g)
$$
has the same properties as those of $F$, with the additional feature that
\begin{equation}\label{eq:mainthereomsuperrefl_Lipschitzpresserved}
  \sup_{x\in X} \|DF_L(x) \|_* = \lip(F_L) = L = \sup_{z\in E} \|G(z)\|_*
\end{equation}
\end{theorem}

\subsection{The necessity part in Theorem \ref{thm:technicalgeneraltheorem}} The ``only if'' part of Theorem  is a consequence of the following lemma. 

\begin{lemma}\label{lem:necessityofCW1omega}
Let $X$ be a normed space and $\omega$ a modulus of continuity. Let $F\in C^{1,\omega}(X)$ be convex and denote $M: = \lip_\omega(DF, X).$ Then $(F,DF)$ satisfies condition \eqref{eq:definitionCw1omega_forallx} with constant $M>0$ on the set $X.$ In other words, $A_{\omega, \co}(F,DF,X) \leq \lip_\omega(DF, X).$  
\end{lemma}
\begin{proof}
By Taylor's theorem and the convexity of $X,$ we can write, for each $x,y\in X,$ 
\begin{align*}
F(x) & = F(y) + DF(y)(x-y) + \int_0^1 (DF(y+t(x-y))-DF(y))(x-y) \ud t \\
&  \leq  F(y) + DF(y)(x-y) + M \int_0^1 \omega(t\|x-y\|) \|x-y \| \ud t \\
& =F(y) + DF(y)(x-y) + M \varphi_\omega(\|x-y\|).
\end{align*}
On the other hand, the convexity of $F$ in $X$ gives, for all $x,z\in x:$
$$
F(x) \geq F(z) + DF(z)(x-z) .
$$
Both inequalities together say that
$$
F(z) + DF(z)(x-z) \leq F(y) + DF(y)(x-y) + M \varphi_\omega(\|x-y\|), \quad x,y,z\in X.
$$
This gives condition \eqref{eq:definitionCw1omega_forallx} with constant $M>0$ on $X.$ 
\end{proof}

\subsection{The sufficiency part of Theorem \ref{thm:technicalgeneraltheorem}} Assume that $\omega, \varphi_\omega$ and $\psi_\omega$ satisfy condition \eqref{eq:assumptiononregularitypsiomega_generaltheorem}. Let $E \subset X$ and $(f,G)$ a $1$-jet on $E$ with the condition \eqref{eq:definitionCw1omega_forallx} for some $M>0.$ We define $g$ as in \eqref{eq:formulaextensionC1w_generaltheorem}. Moreover, define 
\begin{equation}\label{eq:minimalfunction}
m(x) = \sup_{z\in E} \lbrace f(z) + G(z)(x-z) \rbrace, \quad x\in X.
\end{equation}

We split the proof of the sufficiency in Theorem \ref{thm:technicalgeneraltheorem} into two lemmas.

\begin{lemma}\label{lem:properties_g_generaltheorem}
We have that
$$
m(x) \leq g(x), \quad x\in X, \quad \text{and} \quad m(x)= f(x)=g(x), \quad x\in E.
$$
Also, $g$ satisfies the inequality
\begin{equation}\label{eq:smoothness_g_generaltheorem}
\lambda g(x+(1-\lambda)h) + (1-\lambda) g(x-\lambda h) - g(x) \leq \lambda (1-\lambda) \cdot K \cdot M \cdot \varphi_\omega(\|h\|), \quad \lambda \in [0,1], \, z,h \in X.
\end{equation}
Moreover, in the case where $G$ is bounded in $E,$ and $L= \sup_{z\in E} \|G(z)\|_*,$ the function $m$ is $L$-Lipschitz on $X$. 
\end{lemma} 
\begin{proof}
By the inequalities \eqref{eq:definitionCw1omega_forallx} with constant $M>0$ for the jet $(f,G)$ on $E$, we have
$$
f(y) +  G(y)(x-y) + M \varphi_\omega(\|x-y\|) \geq f(z) + G(z) (x-z), \quad y,z\in E,\, x\in X.
$$
The definitions of $g$ and $m$ in \eqref{eq:formulaextensionC1w_generaltheorem} and \eqref{eq:minimalfunction} then implies $m \leq g$ on $X$ and that $m=f=g$ on $E.$

To show \eqref{eq:smoothness_g_generaltheorem}, let $x,h\in X$, $\lambda \in [0,1]$ and $\varepsilon>0.$ By the definition of $g,$ we can find $y\in E$, depending on $\varepsilon$ and $x,$ so that
$$
g(x)+ \varepsilon \geq f(y)+ G(y)(x-y) + M \varphi_\omega(\|x-y\|). 
$$
Then we have
\begin{align*}
\lambda   g(x+(1- & \lambda)h)   + (1-\lambda) g(x-\lambda h) - g(x) \\
& \leq  \lambda \left[ f(y)+ G(y)(x+(1-\lambda)h-y) + M \varphi_\omega(\|x+(1-\lambda)h-y\|) \right] \\
&  \quad + (1-\lambda) \left[ f(y)+ G(y)(x-\lambda h-y) + M \varphi_\omega(\|x-\lambda h-y\|) \right ] \\
& \quad -   \left( f(y)+ G(y)(x-y) + M \varphi_\omega(\|x-y\|) \right) +   \varepsilon \\
& = M \left( \lambda \varphi_\omega(\|(x-y)+(1-\lambda)h\|) + (1-\lambda) \varphi_\omega(\|(x-y)+\lambda h\|) -   \varphi_\omega(\|x-y\|)  \right) +  \varepsilon.
\end{align*}
According to \eqref{eq:assumptiononregularitypsiomega_generaltheorem}, the last term is bounded by
$$
K \cdot M \cdot \varphi_\omega(\|h\|) + \varepsilon.
$$
Since $\varepsilon>0$ is arbitrary, the assertion follows.

In addition, if $G: E \to X^*$ is bounded, and $L= \sup_{z\in E} \|G(z)\|_*$
\end{proof}

\begin{lemma}\label{lem:regularity_convg_general}
The function $F:X \to \R$ from \eqref{eq:formulaextensionC1w_generaltheorem} is convex and of class $C^{1,\omega}(X)$, with $(F,DF)=(f,G)$ on $E$, and satisfies \eqref{eq:estimate_extensionoperator_generaltheorem}.  
\end{lemma}
\begin{proof}
Clearly $m$ is convex and lower semicontinuous, as supremum of convex and continuous functions $X \ni x\mapsto f(z) + G(z)(x-z),$ $z\in E.$ Therefore $F$ is well-defined and, by the same reasons, $F$ is convex and lower-semicontinuous in $X.$ Moreover, by the definition of $F$, one has
$$
m(x) \leq F(x) \leq g(x), \quad x\in X.
$$
By Lemma \ref{lem:properties_g_generaltheorem}, $F(y)=f(y)$ for all $y\in E.$ Using this property and the definitions of $m$ and $g,$ we have that, for all $x\in X$ and $y\in E$ that
\begin{align*}
0 & \leq  \frac{m(x)-f(y)-G(y)(x-y)}{\|x-y\|} \\
& \leq \frac{F(x)-f(y)-G(y)(x-y)}{\|x-y\|} \\
& \leq \frac{g(x)-f(y)-G(y)(x-y)}{\|x-y\|} \leq \frac{\varphi_\omega(\|x-y\|)}{\|x-y\|} .
\end{align*}
Letting $\|x-y\| \to 0^+,$ this implies that $F$ is differentiable at $y\in E$ with $DF(y)=G(y).$

We now prove the regularity properties of $F$. In order to do so, according to Lemma \ref{lem:convexcombregularity_impliesC1omega}, it suffices to prove \eqref{eq:unfiformsmoothnessF_inBanach} for $F.$ We will utilize the following formula for $F:$
$$
F(x) = \inf\left\lbrace \sum_{j=1}^{n}\lambda_{j} g(x_j) \, : \, \lambda_j\geq 0,
\sum_{j=1}^{n}\lambda_j =1, \, x=\sum_{j=1}^{n}\lambda_j x_j, \, n\in\N \right\rbrace, \quad x\in X.
$$
Given $x,h\in X,$ $\lambda \in [0,1]$ and $\varepsilon>0,$ we can find $n\in \N,$ $\lbrace x_j \rbrace_{j=1}^n,$ $\lbrace \lambda_j \rbrace_{j=1}^n$ with $\sum_{j=1}^n \lambda_j=1,$ $\lambda_j \geq 0,$ $x= \sum_{j=1}^n \lambda_j x_j $ and
$$
F(x) \geq \sum_{j=1}^n\lambda_j g(x_j) - \varepsilon
$$
Now, using the fact that $x+(1-\lambda)h = \sum_{j=1}^n \lambda_j (x_j + (1-\lambda) h)$ and $x-\lambda h =\sum_{j=1}^n \lambda_j (x_j -\lambda h),$ we have 
$$
F(x + (1-\lambda) h) \leq \sum_{j=1}^n\lambda_j g(x_j + (1-\lambda) h), \quad \text{and} \quad  F(x  -\lambda h) \leq \sum_{j=1}^n\lambda_j g(x_j -\lambda h).
$$
Using this and Lemma \ref{lem:properties_g_generaltheorem}, we get
\begin{align*}
\lambda F(x+(1-\lambda) h) & + (1-\lambda) F(x-\lambda h) - F(x) \\
& \leq  \sum_{j=1}^n \lambda_j \left[ \lambda g(x_j + (1-\lambda) h) + (1-\lambda) g(x_j - \lambda h) - g(x_j)\right] + \varepsilon \\
& \leq \sum_{j=1}^n \lambda_j \left[\lambda (1-\lambda) \cdot K \cdot M \cdot \varphi_\omega(\|h\|)\right]  + \varepsilon \\
& = \lambda (1-\lambda) \cdot K \cdot M \cdot \varphi_\omega(\|h\|) + \varepsilon.
\end{align*}
Since $\varepsilon >0$ was arbitrary, we have the desired inequality \eqref{eq:unfiformsmoothnessF_inBanach}. 
\end{proof}

Now we prove the part concerning the Lipschitz case in Theorem \ref{thm:technicalgeneraltheorem}.

\begin{lemma}\label{lem:conv_L_Lipschitzcase_generaltheorem}
Assume that $G: E \to X^*$ is bounded with $L:= \sup_{z\in E} \| G(z)\|_*$ and let $F_L:= \conv_L(g),$ as in \eqref{eq:definition_convex_L_envelope}. Then $F_L \in C^{1,\omega}(X)$ is convex and $L$-Lipschitz, with $(F_L,DF_L)=(f,G)$ on $E$ and so that $F_L$ satisfies the estimates \eqref{eq:estimate_extensionoperator_generaltheorem}. 
\end{lemma}
\begin{proof}
We learnt from Lemma \ref{lem:properties_g_generaltheorem} that $m$ is, in addition, $L$-Lipschitz. Therefore $F_L : X \to \R$ is well-defined with $m \leq F_L  \leq g$ on $X.$ In particular, $F_L=g$ on $E.$ Following an identical argument to that of the proof of Lemma \ref{lem:regularity_convg_general}, we have that $F_L$ is differentiable at each $y\in E,$ with $DF_L = G$ on $E.$ Besides, $F_L,$ as a finite supremum of convex and $L$-Lipschitz functions in $X$, is convex and $L$-Lipschitz as well. Thus, it remains to show the regularity properties of $F.$

\end{proof}

\subsection{The superreflexive case: proof of Theorem \ref{thm:mainsuperreflexive}}

Let $(X,\|\cdot\|)$ be a Banach space satisfying \eqref{eq:modulussmoothnessalpha} for some $0 < \alpha \leq 1$ and $C>0.$ Let $\omega$ be a modulus of continuity with the property \eqref{eq:assumptiononOmegaAlpha} for that $\alpha.$ In the next lemma, we state the regularity of the function $\varphi_\omega \circ \| \cdot \|$ in $X$ with an estimate for $\lip_\omega(\psi_\omega, X)$ in terms of $\alpha$ and $C.$

\begin{lemma}\label{lem:regularityofpsi_w}
Define $\psi_\omega = \varphi_\omega \circ \| \cdot \|.$ Then $\psi_\omega \in C^{1,\omega}(X)$ with
$$
\| D \psi_\omega(x)- D\psi_\omega (y)\|_* \leq K \omega(\|x-y\|), \quad x,y\in X,
$$
where $K$ is of the form $K= C \cdot K(\alpha);$ where $C$ is that of \eqref{eq:modulussmoothnessalpha} and $K(\alpha)$ only depends on $\alpha.$ Consequently,
$$
\lambda \psi_\omega(z+(1-\lambda) h) + (1-\lambda) \psi_\omega(z-\lambda h) - \psi_\omega(z) \leq   \lambda (1-\lambda) \cdot K  \cdot \varphi_\omega(\|h\|)
$$ 
for all $\lambda\in [0,1]$, and $z,h \in X.$ 
\end{lemma}
\begin{proof}
With an identical argument to that in the proof of Lemma \ref{lem:regularity_psiomega_Hilbert}, the inequality \eqref{eq:modulussmoothnessalpha} can be rewritten as
$$
\lambda \|x+(1-\lambda)h \|^{1+\alpha} + (1-\lambda) \|x-\lambda h\|^{1+\alpha} - \|x\|^{1+\alpha} \leq C \cdot \lambda (1-\lambda) \cdot K \|h\|^{1+\alpha}, \quad \lambda \in [0,1], \, z,h \in X,
$$
or as
$$
\lambda \psi_\alpha(x+(1-\lambda)h) + (1-\lambda) \psi_\alpha(x-\lambda h) - \psi_\alpha (x) \leq C \cdot \lambda (1-\lambda)  \varphi_\alpha(\|h\|), \quad \lambda \in [0,1], \, z,h \in X;
$$
where $\varphi_\alpha(t) = \frac{t^{1+\alpha}}{1+\alpha}$ and $\psi_\alpha = \varphi_\alpha \circ \| \cdot \|.$ Applying Theorem \ref{thm:relationsAseminorm_Lip}, we get that $\psi_\alpha \in C^{1,\alpha}(X)$ with
$$
\lip_\alpha( D \psi_\alpha , X) \leq \left(\frac{1 + \alpha}{2\alpha}\right)^\alpha \cdot C.
$$
Following the argument in the proof of \cite[Lemma 3.6]{AM21}, we derive
$$
\norm{D \psi_\omega(x)- D\psi_\omega (y)}_* \leq K \omega(\norm{x-y}), \quad \text{where } \:  K: = 1+  4^\alpha \left(\frac{1 + \alpha}{2\alpha}\right)^\alpha \cdot C.
 $$
Also, for $\lambda \in [0,1],$ $z,h\in X,$ the Fundamental Theorem of Calculus yield
$$
\psi_\omega(z+(1-\lambda) h) -\psi_\omega(z) = \int_0^1 D\psi_\omega(z+ t(1-\lambda)h))((1-\lambda)h) \ud t,
$$
and
$$
\psi_\omega(z-\lambda  h) -\psi_\omega(z) = \int_0^1 D\psi_\omega(z- t \lambda h))(-\lambda h) \ud t.
$$
Multiplying the first identity by $\lambda$ and the second one by $(1-\lambda),$ and adding the resulting identities, we get
\begin{align*}
\lambda \psi_\omega(z+(1-\lambda) h) & + (1-\lambda) \psi_\omega(z-\lambda h) - \psi_\omega(z) \\
&  = \lambda (1-\lambda) \int_0^1 \left[ D\psi_\omega(z+ t(1-\lambda)h))-D\psi_\omega(z- t \lambda h)) \right](h) \ud t \\
& \leq \lambda (1-\lambda) K \int_0^1 \omega(t\|h\|) \|h\| \ud t =  \lambda (1-\lambda) K \varphi_\omega(\|h\|).
\end{align*}

\end{proof}

We are now ready to prove Theorem \ref{thm:mainsuperreflexive}.

\begin{proof}[Proof of Theorem \ref{thm:mainsuperreflexive}] We have from Lemma \ref{lem:regularityofpsi_w} that, if we define $\psi_\omega = \varphi_\omega \circ \| \cdot \|,$
$$
\lambda \psi_\omega(z+(1-\lambda) h) + (1-\lambda) \psi_\omega(z-\lambda h) - \psi_\omega(z) \leq   \lambda (1-\lambda) \cdot C \cdot K(\alpha)  \cdot \varphi_\omega(\|h\|)
$$ 
for all $\lambda\in [0,1]$, and $z,h \in X.$ This is precisely the inequality \eqref{eq:assumptiononregularitypsiomega_generaltheorem} with $K= C \cdot K(\alpha)$ in Theorem \ref{thm:technicalgeneraltheorem}.

If we let $E \subset X$ and $(f,G)$ a jet in $E,$ Then there exists $F\in C^{1,\omega}(X)$ convex with $(F,D F) = (f,G)$ on $E$ if and only if $(f,G)$ satisfies the condition \eqref{eq:definitionCw1omega_forallx} for some $M>0.$ Moreover, the formula
\begin{equation*}
F:= \conv(g), \quad g(x) = \inf_{y\in E} \lbrace f(y) + G(y)(x-y) + M \varphi_\omega(\|x-y\|) \rbrace, \quad x\in X
\end{equation*}
defines such an extension, and
$$
\lip_\omega(DF,X) \leq \frac{4}{3} \cdot K \cdot M = \frac{4}{3} \cdot C \cdot K(\alpha) \cdot M.
$$
Furthermore, if in addition $G: E \to X^*$ is bounded and $L:=\sup_{z\in E} \|G(z)\|_*$, then the function
$$
F_L:= \conv_L(g)
$$
has the same properties as those of $F$, with the additional feature that
\begin{equation*}
  \sup_{x\in X} \|DF_L(x) \|_* = \lip(F_L) = L = \sup_{z\in E} \|G(z)\|_*
\end{equation*}
Which is exactly our result from Theorem \ref{thm:mainsuperreflexive}.
\end{proof}

\subsection{The Hilbert case: proof of Theorems \ref{thm:mainHilbert} and \ref{thm:mainHilbertHoldercase}}\label{sect:Hilbertcasebounds}

We will employ Theorem \ref{thm:technicalgeneraltheorem}, with suitable constants $K>0$ in the inequality \eqref{eq:assumptiononregularitypsiomega_generaltheorem}.

\begin{proof}[Proof of Theorem \ref{thm:mainHilbert}] Assume now that $(X, | \cdot|)$ is a Hilbert space and $\omega$ is an increasing modulus of continuity with $\omega(\infty)=\infty$. By Lemma \ref{lem:regularity_psiomega_Hilbert}, the inequality \eqref{eq:assumptiononregularitypsiomega_generaltheorem} holds with constant $K=2$ for $\psi_\omega$.

According to Lemma \ref{lem:rewriteCW1omega}, conditions \eqref{eq:definitionCw1omega_forallx} and \eqref{eq:definitionCw1omega} are fully identical with the same constant $M>0$ for a jet $(f,G)$ on a subset $E$ of $X.$ Therefore, the necessity and sufficiency of the extendability condition \eqref{eq:definitionCw1omega} in Theorem \ref{thm:mainHilbert} follow from those of Theorem \ref{thm:technicalgeneraltheorem}. Moreover, Theorem \ref{thm:technicalgeneraltheorem} can be applied with $K=2$, thus obtaining the desired extensions via formula \eqref{eq:formulaextensionC1w_generaltheorem}, and with the bounds
\begin{equation*}
A_{\omega, \co}(F,\nabla F,X) \leq 2 \cdot M  \quad \text{and} \quad \lip_\omega(\nabla F,X) \leq \frac{8}{3} \cdot M,
\end{equation*}
provided the initial jet $(f,G)$ satisfies condition \eqref{eq:definitionCw1omega_forallx} with constant $M>0$ on $E.$ The part of Theorem \ref{thm:mainHilbert} that concerns the preservation of the Lipschitz constant, provided $G$ is bounded, follows immediately from Theorem \ref{thm:technicalgeneraltheorem}.

\end{proof}

We now give the proof of Theorem \ref{thm:mainHilbertHoldercase}, where we obtain better bounds for $A_{\omega, \co}(F,\nabla F,X)$ for H\"older moduli of continuity.

\begin{proof}[Proof of Theorem \ref{thm:mainHilbertHoldercase}]
If we let $\omega(t) = t^\alpha$, with $0 < \alpha < 1,$ we get from the second part of Lemma \ref{lem:regularity_psiomega_Hilbert} that the function $ \psi_\omega$ satisfies the inequality \eqref{eq:assumptiononregularitypsiomega_generaltheorem} with constant $K=2^{1-\alpha}.$ Since $\omega$ is increasing with $\omega(\infty)= \infty,$ the conditions \eqref{eq:definitionCw1omega_forallx} and \eqref{eq:definitionCw1omega} are identical, and, in turn, identical to \eqref{eq:definitionCW1alpha}. The necessity and sufficienty of the condition \eqref{eq:definitionCW1alpha} then follows from this observation and Theorem \ref{thm:technicalgeneraltheorem}. Moreover, taking into account that $\varphi_\omega(t) = \int_0^ts^\alpha \ud s = \frac{1}{1+\alpha}t^{1 + \alpha}$, Theorem \ref{thm:technicalgeneraltheorem} tells us that the function $F :X \to \R$ given by the formula
$$ 
F:= \conv(g), \quad g(x) = \inf_{y\in E} \lbrace f(y) + G(y)(x-y) + \frac{M}{1 + \alpha}|x-y|^{1 + \alpha}) \rbrace, \quad x\in X,
$$
defines such an extension; where $M>0$ is so that $(f,G)$ satisfies \eqref{eq:definitionCW1alpha} with constant $M.$ Moreover, Theorem \ref{thm:technicalgeneraltheorem} also provides the following bounds for the extension $(F,\nabla F)$ of $(f,G)$:
$$
A_{\alpha, \co}(F,\nabla F,X) \leq 2^{1 - \alpha} M, \quad \lip_\omega(\nabla F,X) \leq 2^{1 - \alpha}\left(\frac{1 + \alpha}{2\alpha}\right)^\alpha M.
$$
And, as before we get the desired Lipschitz properties from Theorem \ref{thm:technicalgeneraltheorem}, in the case where $G : E \to X$ is a bounded map. 

\end{proof}

\section{Lipschitz-preserving $C^1$ convex extensions: proof of Theorem \ref{thm:C1optimal}}\label{sect:C1case}

Let $X$ be a superreflexive space, and let $ \alpha \in (0,1]$ be as in \eqref{eq:modulussmoothnessalpha}. Also, let $E \subset X$ be a compact set, and $(f,G) : E \to \R \times X^*$ be a $1$-jet with $G$ continuous on $E$ and so that $(f,G)$ satisfies conditions \eqref{eq:condition(C)} and \eqref{eq:condition(CW1)}. Also, denote
$$
L : = \sup_{z\in E} \|G(z)\|_*.
$$
As we observed right after Theorem \ref{thm:C1optimal}, the function $f : E  \to \R$ is $L$-Lipschitz on $E.$ The idea is constructing a modulus of continuity $\omega$ with the property \eqref{eq:assumptiononOmegaAlpha} and so that $(f,G)$ satisfies condition \eqref{eq:definitionCw1omega_forallx} for some constant $M>0$ on $E.$ In order to do that, we start defining
\begin{equation}\label{eq:definitionpremodulus_delta}
\delta(t) = \sup\left\lbrace \frac{f(z) + G(z)(x-z) - f(y)-G(y)(x-y)}{\|x-y\|} \,  : \, 0 < \|x-y\| \leq t, \, x\in X, \, y,z\in E \right\rbrace, \quad t\geq 0. 
\end{equation}

\begin{lemma}\label{lem:premodulusdelta}
We have $0 \leq \delta(t) \leq 2L$ for all $t\geq 0,$ $\delta$ is non-decreasing, and $\lim_{t \to 0^+} \delta(t) =0.$ 
\end{lemma}
\begin{proof}
Given $y,z\in E$ we have that $f(y) \geq f(z)+ G(z)(y-z)$ according to condition \eqref{eq:condition(C)}. For every $x\in X,$ we have the estimates
\begin{align*}
f(z) + G(z)(x-z) - f(y)-G(y)(x-y) &  = f(z)  + G(z)(x-z) - f(z)-G(z)(y-z) -  G(y)(x-y)  \\
& = (G(z)-G(y))(x-y) \\
& \leq \|G(z)-G(y)\|_* \| x-y\| \\
& \leq 2 L \|x-y\|.
\end{align*}
Thus $\delta(t) \leq 2 L$ for all $t \geq 0.$ Also, if we take $z=y$ in the supremum defining $\delta(t),$ we clearly get $\delta(t) \geq 0.$

Finally, assume, for the sake of contradiction that there is a sequence $\lbrace t_n \rbrace_n \downarrow 0$ and $\varepsilon>0$ with $\delta(t_n) > \varepsilon$ for all $n\in \N.$ By the definition of $\delta(t_n) ,$ we can find points $x_n \in X,$ $y_n,z_n \in E$ with $x_n \neq y_n$ for all $n\in \N$ and so that
\begin{equation}\label{eq:approximation_premodgoesto0}
t_n = \|x_n-y_n\| \quad \text{and} \quad  f(z_n) + G(z_n)(x_n-z_n) - f(y_n)-G(y_n)(x_n-y_n)  >    t_n  \varepsilon, \quad n \in \N. 
\end{equation}
By the compactness of $E,$ passing to subsequences we may assume that $\lbrace y_n \rbrace_n$ and $\lbrace z_n \rbrace_n$ converge to points $y,z\in E$ respectively. Since $\lbrace t_n  \rbrace \to 0,$ one has that $\lbrace x_n \rbrace_n$ converges to $y$ as well. Moreover, taking limits in the inequality above and using the continuity of both $f$ and $G,$ we arrive that
$$
f(z)+G(z)(y-z) - f(y) \geq 0.
$$
But condition \eqref{eq:condition(C)} then implies that $f(y) = f(z) + G(z)(y-z).$ Using the other condition \eqref{eq:condition(CW1)}, we get that $G(y)=G(z).$ But using again \eqref{eq:condition(C)} we have that
$$
f(y_n) \geq f(z_n) + G(z_n) (y_n-z_n), \quad n\in \N,
$$ 
and inserting this into \eqref{eq:approximation_premodgoesto0} we arrive at
\begin{align*}
\varepsilon & < \frac{f(z_n) + G(z_n)(x_n-z_n) - f(y_n)-G(y_n)(x_n-y_n)}{t_n} \leq \frac{(G(z_n)-G(y_n))(x_n-y_n)}{t_n} \\
& \leq \frac{\|G(z_n)-G(y_n)\|_* \|x_n-y_n\|}{t_n} = \|G(z_n)-G(y_n)\|_*.
\end{align*} 
for all $n\in \N.$ The limit of the last term as $n \to \infty$ is, by the continuity of $G,$ precisely $\|G(z)-G(y\| = 0.$ This is a contradiction.

  We may conclude that $\lim_{t \to 0^+} \delta(t)=0.$ 
\end{proof}
 
We now construct a honest modulus of continuity between $\delta$ and $2L.$ The construction is almost the same as in the proof of \cite[Theorem 1]{AP56}.

\begin{lemma}\label{lem:constuctionmodulus}
There exists a modulus of continuity $\Delta : [0, \infty) \to [0, \infty)$ with $\delta(t) \leq \Delta(t) \leq 2L$ for all $t \geq 0.$ 
\end{lemma}
\begin{proof}
Define
$$
\delta_1(t) : = \inf_{0<s<1} \left\lbrace \delta(s) + \frac{2L}{s} t \right\rbrace, \quad t \geq 0. 
$$
Since $\delta$ is nonnegative, $\delta_1$ is nonnegative too. Also, $\delta_1$ is an infimum of a family of non-decreasing and concave (affine) functions, and so $\delta_1$ is non-decreasing and concave as well.

Also, if $t_n \downarrow 0$, we can select $s_n= \sqrt{t_n}$ in the infimum, thus obtaining
$$
\delta_1(t_n) \leq \delta(s_n) + \frac{2L}{s_n} t_n  = \delta(s_n) +  2L  s_n.
$$
By Lemma \ref{lem:premodulusdelta}, we have that $\delta(0^+)=0$, and so the last term tends to $0.$ This shows that $\delta_1(0^+)$, and thus $\delta_1$ satisfies all the properties of a modulus of continuity.

We now want to prove that $\delta(t) \leq \delta_1(t)$ for all $t \geq 0$. In the case where $t \geq 1$, the inequality follows from the fact that $\delta(t) \leq 2L.$ And if $t < 1,$ we again use the same fact to get
$$
\delta_1(t) \geq  \inf_{0<s<1} \left\lbrace \delta(s) + \frac{t}{s} \delta(t) \right\rbrace.
$$
Studying separately the cases $s < t$ and $s \geq t$ in the infimum, it is clear that the last term is not smaller than $\delta(t),$ as desired.

After this, one can simply set $\Delta(t) :=\min \lbrace 2L, \delta_1(t) \rbrace$ for all $t \geq 0.$ Since $\delta(t) \leq 2L$ for all $t \geq 0,$ clearly $\delta(t) \leq \Delta(t) \leq 2L.$ Also, $\Delta$ is a concave, non-decreasing, and nonnegative, as a minimum of two functions with those properties. Since $\lim_{t\to 0^+} \delta_1(t)=0,$ it is obvious that the same holds with $\Delta$ in place of $\delta_1.$ 
\end{proof}

We now modify $\Delta$ to obtain a new modulus that, additionally, satisfies \eqref{eq:assumptiononOmegaAlpha}, for the parameter $ \alpha \in (0,1]$ given by the space $X.$

\begin{lemma}\label{lem:constucitonmodulus_slowerthanalpha}
Define
$$
\omega(t) = \left( \Delta(t) \right)^\alpha, \quad t \geq 0.
$$
Then $\omega$ is a modulus of continuity satisfying \eqref{eq:assumptiononOmegaAlpha} and
$$
\Delta(t) \leq    (2L)^{1-\alpha}  \omega(t), \quad t \geq 0.
$$
\end{lemma}

\begin{proof}
To show that $\omega$ is a modulus of continuity, the only non-trivial property is perhaps that $\omega$ is concave. This is easy and follows from the fact that $\Delta$ is concave, and that $[0, \infty) \ni r \to r^\alpha$ is concave and non-decreasing. We include the short argument for the sake of completeness. If $\lambda \in [0,1]$ and $t_1,t_2 \geq 0$, we write
\begin{align*}
    \omega\left( \lambda t_1 + (1-\lambda) t_2 \right) &  =  \left( \Delta( \lambda t_1 + (1-\lambda) t_2 ) \right)^\alpha \geq  \left( \lambda \Delta(t_1) + (1-\lambda) \Delta(t_2)\right)^\alpha \\
    & \geq  \lambda \left( \Delta(t_1) \right)^\alpha  +  (1-\lambda) \left( \Delta(t_2) \right)^\alpha = \lambda\omega(t_1) + (1 - \lambda)\omega(t_2).
\end{align*}
To show that this satisfies \eqref{eq:assumptiononOmegaAlpha}, note that $t \mapsto t/\Delta(t)$ is non-decreasing, as $\Delta$ is a modulus of continuity; see Proposition \ref{prop:inequalities_omega_varphiomega}. Since $\omega= \Delta^\alpha$, we then get that $t \mapsto t^\alpha/ \omega(t)$ is non-decreasing too. This is precisely condition \eqref{eq:assumptiononOmegaAlpha} for $\omega.$

Finally, the inequality easy follows from the definition of $\omega$ and the fact that $\Delta(t) \leq 2L$ for all $t \geq 0:$
$$
\Delta(t) = (\Delta(t))^{1-\alpha}  (\Delta(t))^\alpha \leq (2L)^{1-\alpha} \omega(t), \quad t \geq 0.
$$

\end{proof}

\begin{lemma}\label{lem:jetsatisfiesCW1w_incaseC1}
Let $\omega$ be the modulus defined in Lemma \ref{lem:constucitonmodulus_slowerthanalpha}. Then $(f,G)$ satisfies condition \eqref{eq:definitionCw1omega_forallx} with some constant $M>0$ on $E$.
\end{lemma}
\begin{proof}
We need to find $M>0$ for which
$$
f(z) + G(z)(x-z) - f(y)-G(y)(x-y) \leq M \varphi_\omega(\|x-y\|), \quad x\in X, \, y,z \in E.
$$
Let $x\in X,$ $y,z\in E $ and $M>0$ a large enough number to be fixed later. By the definition \eqref{eq:definitionpremodulus_delta} and Lemma \ref{lem:constuctionmodulus} one has
$$
f(z) + G(z)(x-z) - f(y)-G(y)(x-y) \leq \|x-y\| \delta(\|x-y\|) .
$$

From Lemma \ref{lem:constuctionmodulus} we have 
$$
f(z) + G(z)(x-z) - f(y)-G(y)(x-y) \leq \|x-y\| \Delta(\|x-y\|) .
$$
Subsisting our inequality from Lemma \ref{lem:constucitonmodulus_slowerthanalpha} we have 
$$ 
f(z) + G(z)(x-z) - f(y)-G(y)(x-y)  \leq \|x-y\| (2L)^{1-\alpha}\omega(\|x-y\|) .
$$ 
From the second property in Proposition \ref{prop:inequalities_omega_varphiomega} we have 
$$
f(z) + G(z)(x-z) - f(y)-G(y)(x-y) \leq 2(2L)^{1 - \alpha} \varphi_{\omega}(\|x-y\|).
$$
This completes the proof with the constant $M = 2(2L)^{1 - \alpha}.$
\end{proof}

We are now ready to prove Theorem \ref{thm:C1optimal}. 

\begin{proof}[Proof of Theorem \ref{thm:C1optimal}]
Let us first prove the necessity, that is, given $F\in C^1(X)$ and convex and let us prove the conditions \eqref{eq:condition(C)} and \eqref{eq:condition(CW1)} for the jet $(F, DF)$ for arbitrary $x,y\in X.$ By convexity and differentiability of $F,$ one has
$$
F(x) \geq F(y) + DF(y)(x-y), \quad x,y\in X.
$$
Also, if $F(x) = F(y) + DF(y)(x-y)$, we have, for all $z\in X$ that
$$
F(z) \geq F(y) + DF(y)(z-y) = F(x) + DF(y)(z-x).
$$
This shows that $DF(y) \in X^*$ belongs to the subdifferential $\partial F(x)$ of $F$ at the point $x.$ But since $F$ is differentiable (at $x$), $\partial F(x) = \lbrace DF(x) \rbrace,$ and this shows that necessarily $DF(y)= DF(x),$ as desired.

For the sufficiency, Lemma \ref{lem:jetsatisfiesCW1w_incaseC1} allows to apply Theorem \ref{thm:mainsuperreflexive} to the jet $(f,G)$ on $E.$ This completes the proof of Theorem \ref{thm:C1optimal}.
\end{proof}

\section*{Acknowledgements}

Carlos Mudarra was supported by the Marie Skłodowska-Curie (MSCA-EF) European Fellowship, grant number 101151594, from the Horizon Europe Funding program. He also acknowledges financial support from the Research Council of Norway via the project ``Fourier Methods and Multiplicative Analysis'', grant number 334466. 

\medskip

\bibliographystyle{amsplain}
\bibliography{main}

@article {AP56,
    AUTHOR = {Aronszajn, N. and Panitchpakdi, P.},
     TITLE = {Extension of uniformly continuous transformations and
              hyperconvex metric spaces},
   JOURNAL = {Pacific J. Math.},
  FJOURNAL = {Pacific Journal of Mathematics},
    VOLUME = {6},
      YEAR = {1956},
     PAGES = {405--439},
      ISSN = {0030-8730,1945-5844},
}

@article {AM21,
    AUTHOR = {Azagra, D. and Mudarra, C.},
     TITLE = {{$C^{1,\omega}$} extension formulas for 1-jets on {H}ilbert
              spaces},
   JOURNAL = {Adv. Math.},
  FJOURNAL = {Advances in Mathematics},
    VOLUME = {389},
      YEAR = {2021},
     PAGES = {Paper No. 107928, 44},
       DOI = {10.1016/j.aim.2021.107928},
       URL = {https://doi.org/10.1016/j.aim.2021.107928},
}

@article {AM20,
    AUTHOR = {Azagra, D. and Mudarra, C.},
     TITLE = {Convex {$C^1$} extensions of 1-jets from compact subsets of
              {H}ilbert spaces},
   JOURNAL = {C. R. Math. Acad. Sci. Paris},
  FJOURNAL = {Comptes Rendus Math\'ematique. Acad\'emie des Sciences. Paris},
    VOLUME = {358},
      YEAR = {2020},
    NUMBER = {5},
     PAGES = {551--556},
       DOI = {10.5802/crmath.62},
       URL = {https://doi.org/10.5802/crmath.62},
}

@article {AM19APDE,
    AUTHOR = {Azagra, D. and Mudarra, C.},
     TITLE = {Global geometry and {$C^1$} convex extensions of 1-jets},
   JOURNAL = {Anal. PDE},
  FJOURNAL = {Analysis \& PDE},
    VOLUME = {12},
      YEAR = {2019},
    NUMBER = {4},
     PAGES = {1065--1099},
       DOI = {10.2140/apde.2019.12.1065},
       URL = {https://doi.org/10.2140/apde.2019.12.1065},
}

@article {ALeGM18,
    AUTHOR = {Azagra, D. and Le Gruyer, E. and Mudarra, C.},
     TITLE = {Explicit formulas for {$C^{1,1}$} and {$C_{\rm
              conv}^{1,\omega}$} extensions of 1-jets in {H}ilbert and
              superreflexive spaces},
   JOURNAL = {J. Funct. Anal.},
  FJOURNAL = {Journal of Functional Analysis},
    VOLUME = {274},
      YEAR = {2018},
    NUMBER = {10},
     PAGES = {3003--3032},
       DOI = {10.1016/j.jfa.2017.12.007},
       URL = {https://doi.org/10.1016/j.jfa.2017.12.007},
}

@article {AM17,
    AUTHOR = {Azagra, D. and Mudarra, C.},
     TITLE = {Whitney extension theorems for convex functions of the classes
              {$C^1$} and {$C^{1,\omega}$}},
   JOURNAL = {Proc. Lond. Math. Soc. (3)},
  FJOURNAL = {Proceedings of the London Mathematical Society. Third Series},
    VOLUME = {114},
      YEAR = {2017},
    NUMBER = {1},
     PAGES = {133--158},
       DOI = {10.1112/plms.12006},
       URL = {https://doi.org/10.1112/plms.12006},
}

@book {BV10,
    AUTHOR = {Borwein, J. M. and Vanderwerff, J. D.},
     TITLE = {Convex functions: constructions, characterizations and
              counterexamples},
    SERIES = {Encyclopedia of Mathematics and its Applications},
    VOLUME = {109},
 PUBLISHER = {Cambridge University Press, Cambridge},
      YEAR = {2010},
     PAGES = {x+521},
      ISBN = {978-0-521-85005-6},
       DOI = {10.1017/CBO9781139087322},
       URL = {https://doi.org/10.1017/CBO9781139087322},
}

@article {BS01,
    AUTHOR = {Brudnyi, Y. and Shvartsman, P.},
     TITLE = {Whitney's extension problem for multivariate
              {$C^{1,\omega}$}-functions},
   JOURNAL = {Trans. Amer. Math. Soc.},
  FJOURNAL = {Transactions of the American Mathematical Society},
    VOLUME = {353},
      YEAR = {2001},
    NUMBER = {6},
     PAGES = {2487--2512},
       DOI = {10.1090/S0002-9947-01-02756-8},
       URL = {https://doi.org/10.1090/S0002-9947-01-02756-8},
}

@book {DGZ93,
    AUTHOR = {Deville, R. and Godefroy, G. and Zizler, V.},
     TITLE = {Smoothness and renormings in {B}anach spaces},
    SERIES = {Pitman Monographs and Surveys in Pure and Applied Mathematics},
    VOLUME = {64},
 PUBLISHER = {Longman Scientific \& Technical, Harlow; copublished in the
              United States with John Wiley \& Sons, Inc., New York},
      YEAR = {1993},
     PAGES = {xii+376},
      ISBN = {0-582-07250-6},
}

@article {F05,
    AUTHOR = {Fefferman, C.},
     TITLE = {A sharp form of {W}hitney's extension theorem},
   JOURNAL = {Ann. of Math. (2)},
  FJOURNAL = {Annals of Mathematics. Second Series},
    VOLUME = {161},
      YEAR = {2005},
    NUMBER = {1},
     PAGES = {509--577},
      ISSN = {0003-486X,1939-8980},
       DOI = {10.4007/annals.2005.161.509},
       URL = {https://doi.org/10.4007/annals.2005.161.509},
}

@article {F06,
    AUTHOR = {Fefferman, C.},
     TITLE = {Whitney's extension problem for {$C^m$}},
   JOURNAL = {Ann. of Math. (2)},
  FJOURNAL = {Annals of Mathematics. Second Series},
    VOLUME = {164},
      YEAR = {2006},
    NUMBER = {1},
     PAGES = {313--359},
      ISSN = {0003-486X,1939-8980},
       DOI = {10.4007/annals.2006.164.313},
       URL = {https://doi.org/10.4007/annals.2006.164.313},
}

@article {FIL14,
    AUTHOR = {Fefferman, C. and Israel, A. and Luli, G.K.},
     TITLE = {Sobolev extension by linear operators},
   JOURNAL = {J. Amer. Math. Soc.},
  FJOURNAL = {Journal of the American Mathematical Society},
    VOLUME = {27},
      YEAR = {2014},
    NUMBER = {1},
     PAGES = {69--145},
      ISSN = {0894-0347,1088-6834},
       DOI = {10.1090/S0894-0347-2013-00763-8},
       URL = {https://doi.org/10.1090/S0894-0347-2013-00763-8},
}

@article {G58,
    AUTHOR = {Glaeser, G.},
     TITLE = {\'{E}tude de quelques alg\`ebres tayloriennes},
   JOURNAL = {J. Analyse Math.},
  FJOURNAL = {Journal d'Analyse Math\'ematique},
    VOLUME = {6},
      YEAR = {1958},
 NUMBER = {2},
     PAGES = {1--124},
      ISSN = {0021-7670,1565-8538},
       DOI = {10.1007/BF02790231},
       URL = {https://doi.org/10.1007/BF02790231},
}

@article {JZ25,
    AUTHOR = {Johanis, M. and Zaj\'{i}\v{c}ek, L.},
     TITLE = {On {$C^1$} {W}hitney extension theorem in {B}anach spaces},
   JOURNAL = {J. Funct. Anal.},
  FJOURNAL = {Journal of Functional Analysis},
    VOLUME = {289},
      YEAR = {2025},
    NUMBER = {9},
     PAGES = {Paper No. 111061, 34},
      ISSN = {0022-1236,1096-0783},
       DOI = {10.1016/j.jfa.2025.111061},
       URL = {https://doi.org/10.1016/j.jfa.2025.111061},
}

@article {LeG09,
    AUTHOR = {Le Gruyer, E.},
     TITLE = {Minimal {L}ipschitz extensions to differentiable functions
              defined on a {H}ilbert space},
   JOURNAL = {Geom. Funct. Anal.},
  FJOURNAL = {Geometric and Functional Analysis},
    VOLUME = {19},
      YEAR = {2009},
    NUMBER = {4},
     PAGES = {1101--1118},
      ISSN = {1016-443X,1420-8970},
       DOI = {10.1007/s00039-009-0027-1},
       URL = {https://doi.org/10.1007/s00039-009-0027-1},
}

@article {JSSG13,
    AUTHOR = {Jim\'{e}nez-Sevilla, M. and S\'{a}nchez-Gonz\'{a}lez, L.},
     TITLE = {On smooth extensions of vector-valued functions defined on
              closed subsets of {B}anach spaces},
   JOURNAL = {Math. Ann.},
  FJOURNAL = {Mathematische Annalen},
    VOLUME = {355},
      YEAR = {2013},
    NUMBER = {4},
     PAGES = {1201--1219},
 ISSN = {0025-5831,1432-1807},
DOI = {10.1007/s00208-012-0814-0},
       URL = {https://doi.org/10.1007/s00208-012-0814-0},
 NOTE= {corrigendum in Math. Ann. \textbf{392} (2025), no. 2, 2969--2979},
}

@article {P75,
    AUTHOR = {Pisier, G.},
     TITLE = {Martingales with values in uniformly convex spaces},
   JOURNAL = {Israel J. Math.},
  FJOURNAL = {Israel Journal of Mathematics},
    VOLUME = {20},
      YEAR = {1975},
    NUMBER = {3-4},
     PAGES = {326--350},
      ISSN = {0021-2172},
       DOI = {10.1007/BF02760337},
       URL = {https://doi.org/10.1007/BF02760337},
}

@article {S17,
    AUTHOR = {Shvartsman, P.},
     TITLE = {Whitney-type extension theorems for jets generated by
              {S}obolev functions},
   JOURNAL = {Adv. Math.},
  FJOURNAL = {Advances in Mathematics},
    VOLUME = {313},
      YEAR = {2017},
     PAGES = {379--469},
      ISSN = {0001-8708,1090-2082},
       DOI = {10.1016/j.aim.2017.04.009},
       URL = {https://doi.org/10.1016/j.aim.2017.04.009},
}

@article {VNC78,
    AUTHOR = {Vladimirov, A. A. and Nesterov, Ju. E. and \v{C}ekanov, Ju. 
              N.},
     TITLE = {Uniformly convex functionals},
   JOURNAL = {Vestnik Moskov. Univ. Ser. XV Vychisl. Mat. Kibernet.},
  FJOURNAL = {Vestnik Moskovskogo Universiteta. Seriya XV.
              Vychislitel\cprime naya Matematika i Kibernetika},
      YEAR = {1978},
    NUMBER = {3},
     PAGES = {12--23},
      ISSN = {0137-0782},
}

@article {W73,
    AUTHOR = {Wells, J. C.},
     TITLE = {Differentiable functions on {B}anach spaces with {L}ipschitz
              derivatives},
   JOURNAL = {J. Differential Geometry},
  FJOURNAL = {Journal of Differential Geometry},
    VOLUME = {8},
      YEAR = {1973},
     PAGES = {135--152},
      ISSN = {0022-040X,1945-743X},
}

@article {W34,
    AUTHOR = {Whitney, H.},
     TITLE = {Analytic extensions of differentiable functions defined in
              closed sets},
   JOURNAL = {Trans. Amer. Math. Soc.},
  FJOURNAL = {Transactions of the American Mathematical Society},
    VOLUME = {36},
      YEAR = {1934},
    NUMBER = {1},
     PAGES = {63--89},
      ISSN = {0002-9947,1088-6850},
       DOI = {10.2307/1989708},
       URL = {https://doi.org/10.2307/1989708},
}

@book {Z02,
    AUTHOR = {Z\u{a}linescu, C.},
     TITLE = {Convex analysis in general vector spaces},
 PUBLISHER = {World Scientific Publishing Co., Inc., River Edge, NJ},
      YEAR = {2002},
     PAGES = {xx+367},
      ISBN = {981-238-067-1},
       DOI = {10.1142/9789812777096},
       URL = {https://doi.org/10.1142/9789812777096},
}

\end{document}